%% file: isd-rev.tex
\documentclass[final]{siamltex}
\pdfoutput=1

\usepackage{amsmath,amsfonts}
\usepackage{graphicx}
\usepackage{url}
\usepackage{subfig}
\usepackage{booktabs}

\DeclareMathAlphabet{\mathpzc}{OT1}{pzc}{m}{it} 

\graphicspath{{./}{./figs/}}

\newcommand{\mA}{\mathbf{A}}
\newcommand{\mC}{\mathbf{C}}
\newcommand{\mCt}{\mathbf{\tilde{C}}}
\newcommand{\mG}{\mathbf{G}}
\newcommand{\mI}{\mathbf{I}}

\newcommand{\mK}{\mathbf{K}}
\newcommand{\mM}{\mathbf{M}}

\newcommand{\mV}{\mathbf{V}}
\newcommand{\mW}{\mathbf{W}}
\newcommand{\mWt}{\mathbf{\tilde{W}}}

\newcommand{\va}{\mathbf{a}}

\newcommand{\vc}{\mathbf{c}}
\newcommand{\vf}{\mathbf{f}}

\newcommand{\vs}{\mathbf{s}}
\newcommand{\vt}{\mathbf{t}}
\newcommand{\vx}{\mathbf{x}}
\newcommand{\vu}{\mathbf{u}}

\newcommand{\vw}{\mathbf{w}}
\newcommand{\vwt}{\mathbf{\tilde{w}}}
\newcommand{\vy}{\mathbf{y}}
\newcommand{\vyt}{\mathbf{\tilde{y}}}
\newcommand{\vz}{\mathbf{z}}
\newcommand{\vzt}{\mathbf{\tilde{z}}}

\newcommand{\sR}{\mathcal{R}}

\newcommand{\sX}{\mathcal{X}}
\newcommand{\sY}{\mathcal{Y}}
\newcommand{\sYt}{\tilde{\mathcal{Y}}}
\newcommand{\sZ}{\mathcal{Z}}

\newcommand{\Exp}[1]{\mathbb{E}\left[#1\right]}
\newcommand{\CExp}[2]{\mathbb{E}\left[#1\,|\,#2\right]}
\newcommand{\Cov}[2]{\text{Cov}\left[#1,\,#2\right]}
\newcommand{\Corr}[2]{\text{Corr}\left[#1,\,#2\right]}
\newcommand{\Var}[1]{\text{Var}\left[#1\right]}

\newcommand{\bmat}[1]{\begin{bmatrix}#1\end{bmatrix}}
\newcommand{\ddx}[2]{\frac{\partial #1}{\partial x_{#2}}}

\newcommand{\trace}[1]{\operatorname{trace}\left(#1\right)}
\newcommand{\verteq}{\rotatebox{90}{$\;\equiv\;$}}

\newtheorem{assumption}{Assumption}

\title{Active subspace methods in theory and practice: applications to kriging surfaces}

\author{ 
Paul G.~Constantine\thanks{Ben L.~Fryrear Assistant Professor
  of Applied Mathematics and Statistics,
  Colorado School of Mines, Golden,
    Colorado 80401 ({\tt paul.constantine@mines.edu}).}  
\and Eric Dow\thanks{Department of Aeronautics and Astronautics, Massachusetts
    Institute of Technology, Cambridge, Massachusetts 02139 ({\tt
      ericdow@mit.edu})}.  
\and Qiqi Wang\thanks{Assistant Professor, Department of
    Aeronautics and Astronautics, Massachusetts Institute of
    Technology, Cambridge, Massachusetts 02139 ({\tt qiqi@mit.edu})}.
}

\begin{document}

\maketitle

\begin{abstract}
Many multivariate functions in engineering models vary primarily along a few directions in the space of input parameters. When these directions correspond to coordinate directions, one may apply global sensitivity measures to determine the most influential parameters. However, these methods perform poorly when the directions of variability are not aligned with the natural coordinates of the input space. We present a method to first detect the directions of the strongest variability using evaluations of the gradient and subsequently exploit these directions to construct a response surface on a low-dimensional subspace---i.e., the \emph{active subspace}---of the inputs. We develop a theoretical framework with error bounds, and we link the theoretical quantities to the parameters of a kriging response surface on the active subspace. We apply the method to an elliptic PDE model with coefficients parameterized by 100 Gaussian random variables and compare it with a local sensitivity analysis method for dimension reduction. 
\end{abstract}


\pagestyle{myheadings}
\thispagestyle{plain}
\markboth{P.~G. CONSTANTINE, E. DOW, AND Q. WANG}{ACTIVE SUBSPACE METHODS}

\input{sec0-intro}
\input{sec1-active}
\input{sec2-approx}
\input{sec3-kriging}
\input{sec4-example}
\input{sec5-conclusion}

\clearpage
\bibliographystyle{siam}
\bibliography{isd}

\end{document}

%% file: sec0-intro.tex
\section{Introduction \& motivation}

As computational models of physical systems become more complex, the need increases for uncertainty quantification (UQ) to enable defensible predictions. Monte Carlo methods are the workhorse of UQ, where model inputs are sampled according to a characterization of their uncertainty and the corresponding model outputs are treated as a data set for statistical analysis. However, the slow convergence of Monte Carlo methods coupled with the high computational cost of the models has led many to employ response surfaces trained on a few carefully selected runs in place of the full model. This strategy has had great success in forward~\cite{giunta2006promise,Li2010} and inverse uncertainty propagation problems~\cite{Marzouk2007,bliznyuk2008bayesian,Constantine11c} as well as
optimization~\cite{Jones2001,wild2008orbit}. However, most response surfaces suffer from the curse of dimensionality, where the cost of constructing an accurate surface increases exponentially as the dimension (i.e., the number of input parameters) increases.

To make construction tractable, one may first perform sensitivity analysis~\cite{Saltelli2008} to determine which variables have the most influence on the model predictions. With a ranking of the inputs, one may construct response surfaces that concentrate the approximation on the most influential variables, e.g., through a suitably anisotropic design; the same concept applies to mesh refinement strategies for solving PDEs. Methods for sensitivity analysis are typically classified as local perturbation or global methods. Local methods perturb the inputs---often along coordinate directions---around a nominal value and measure the effects on the outputs. Though relatively inexpensive, local methods are fraught with difficulties like sensitivity to noise and the choice of the perturbation step. Also, the local sensitivity measured at the nominal condition may be very different elsewhere in the parameter space. Global methods address these issues by providing integrated measures of the output's variability over the full range of parameters; consequently, they are computationally more expensive. Methods based on variance decompositions~\cite{Owen13,Saltelli2008} require approximating high-dimensional integrals in order to rank the inputs.
Consider the simple function $f(x_1,x_2)=\exp(0.7x_1+0.3x_2)$ defined on $[-1,1]^2$ plotted in Figure \ref{fig:expfun}. 
A local perturbation method at the origin with a stepsize $\Delta x=0.1$ reveals $f(\Delta x,0)=1.0725$ and $f(0,\Delta x)=1.0305$ to four digits. The larger effect of the perturbation in $x_1$ designates $x_1$ more important than $x_2$. The global Sobol' indices~\cite{Owen13} for the main (univariate) effects of the ANOVA decomposition are $\sigma_{1}=0.1915$ and $\sigma_2=0.0361$, which yields a similar conclusion regarding $x_1$'s importance. 

\begin{figure}[ht]
\centering
\includegraphics[width=0.5\textwidth]{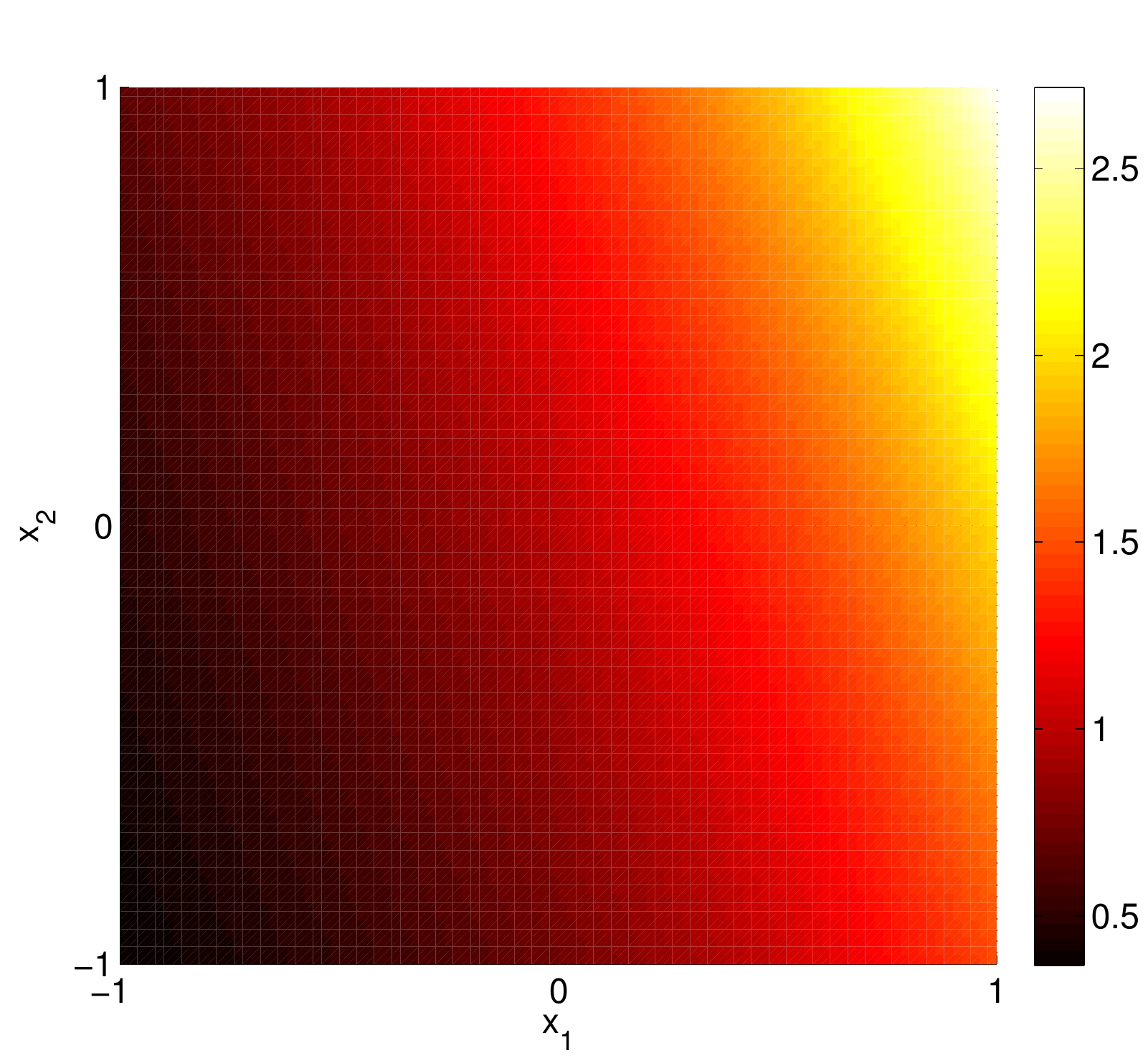}
\caption{The function $f(x_1,x_2) = \exp(0.7x_1+0.3x_2)$ varies strongest along the direction $[0.7, 0.3]$, and it is flat in the direction $[0.3,-0.7]$. (Colors are visible in the electronic version.)
}
\label{fig:expfun}
\end{figure}

Both classes of methods rank the coordinates of the inputs. However, some models may vary most prominently along directions of the input space that are not aligned with the coordinate system. The example $f$ plotted in Figure \ref{fig:expfun} varies strongest along the direction $[0.7,0.3]$, and it is flat along the direction $[-0.3,0.7]$. This bivariate function is in effect univariate once the coordinate system has been rotated appropriately. This suggest an alternative form of dimension reduction: rotate the coordinates such that the directions of the strongest variation are aligned with the rotated coordinates, and construct a response surface using only the most important rotated coordinates. 

We propose a method based on gradient evaluations for detecting and exploiting the directions of strongest variability of a given function to construct an approximation on a low-dimensional subspace of the function's inputs. Given continued interest in gradient computations based on adjoint methods~\cite{Bryson75, Jameson1988} and algorithmic differentiation~\cite{Griewank00}, it is not unreasonable to assume that one has access to the gradient of the function.  We detect the directions by evaluating the function's gradient at a set of input points and determining a rotation of the input space that separates the directions of relative variability from directions of relative flatness.  We exploit these directions by first projecting the input space to the low-dimensional subspace that captures the function's variability and then approximating the function on the subspace. Following Russi's 2010 Ph.D. thesis~\cite{Russi2010}, we call this low-dimensional subspace the \emph{active subspace}.

Subspace approximations are commonly used in optimization, where local quadratic models of a function are decomposed to reveal search directions~\cite{gill1981practical}. They are also found in many areas of model reduction~\cite{Antoulas2005} and optimal control~\cite{Zhou1996}, where a high-dimensional state space vector is approximated by a linear combination of relatively few basis vectors. Common to both of these fields are methods for matrix factorizations and eigenvalue computations~\cite{Saad1992}, which are replete with subspace oriented approaches. The use of subspace methods for approximating high-dimensional functions arising in science and engineering models appears rare by comparison. Recent work by Lieberman, et al~\cite{Lieberman10} describes a method for finding a subspace in a high-dimensional input space via a greedy optimization procedure. Russi~\cite{Russi2010} proposes a method for discovering the \emph{active subspace} of a function using evaluations of the gradient and constructing a quadratic response surface on the subspace; his methodology is similar to ours in practice.  Recently Fornasier, et al~\cite{Fornasier2012} analyzed subspace approximation algorithms that do not need gradient evaluations but make strong assumptions on the function they are approximating; they take advantage of results from compressed sensing. Our previous work has applied the active subspace method to design optimization~\cite{Chen2011,Dow2013}, inverse analysis~\cite{Constantine11c}, and spatial sensitivity~\cite{Sensitive12}.

The contribution of this paper is two-fold. First we provide a theoretical foundation for gradient-based dimension reduction and subspace approximation. We construct and factorize a covariance-like matrix of the gradient to determine the directions of variability. These directions define a new set of coordinates which we separate into a set $\vy$ along which the function varies the strongest and a set $\vz$ along which the function varies relatively little on average. We then approximate the function by a sequence of three functions that are \emph{$\vz$-invariant}, i.e., that are essentially functions of only the $\vy$ coordinates. The first is a theoretical best approximation via conditional expectation. The second approximates the conditional expectation with a Monte Carlo method. The third builds a response surface on the $\vy$ coordinates using a few evaluations of the Monte Carlo approximations. We provide error bounds for these approximations, and we examine the effects of using directions that are slightly perturbed. Second, we provide a bridge between the theoretical analysis and computational practice by (i) relating the derived error bounds to SVD-based approaches for discovering the active subspace and (ii) heuristically linking the theoretical quantities to the parameters of a kriging surface constructed on the active subspace. We apply this procedure to an elliptic PDE model with a 100-parameter model for the coefficients and a scalar quantity of interest. We compare the active subspace approach to a dimension reduction approach based on local sensitivity analysis. 

%% file: sec1-active.tex
\section{Active subspaces and $\vz$-invariance}
In this section we describe the class of functions that vary primarily along a few directions of the input space.  We characterize the active subspace and discuss a computational procedure for approximating its basis. We perform the analysis using tools from probability theory such as expectation $\Exp{\cdot}$, but we emphasize that there is nothing inherently stochastic about the functions or the approximations; the probability notation provides a convenient shorthand.

Consider a function $f$ with $m$ continuous inputs
\begin{equation}
f \;=\; f(\vx), \quad \vx\in\sX\subseteq \mathbb{R}^m,
\end{equation}
where we assume without loss of generality that $\sX$ is centered at the origin. Let $\sX$ be equipped with a bounded probability density function $\rho:\mathbb{R}^m\rightarrow\mathbb{R}_+$, where
\begin{equation}
\label{eq:rho}
\rho(\vx)>0,\;\vx\in\sX\quad\mbox{and}\quad\rho(\vx)=0,\;\vx\not\in\sX. 
\end{equation}
We assume that $f$ is absolutely continuous and square-integrable with respect to $\rho$.  Denote the gradient of $f$ by the column vector
$\nabla_{\vx} f(\vx) \;=\; \bmat{ \ddx{f}{1} & \cdots & \ddx{f}{m} }^T$.
Define the $m\times m$ matrix $\mC$ by
\begin{equation}
\label{eq:C}
\mC \;=\; \Exp{(\nabla_{\vx} f)\,(\nabla_{\vx} f)^T},
\end{equation}
where we assume that $f$ is such that $\mC$ exists; in other words, the products partial derivatives are integrable.  $\mC$ can be interpreted as the uncentered covariance of the gradient vector. Note that $\mC$ is symmetric and positive semidefinite, so it admits a real eigenvalue decomposition
\begin{equation}
\label{eq:gradcovariance}
\mC \;=\; \mW\Lambda\mW^T, 
\quad\Lambda \;=\; \mathrm{diag}\,(\lambda_1,\dots,\lambda_m),
\quad \lambda_1\geq\cdots\geq\lambda_m\geq 0.
\end{equation}
The following lemma quantifies the relationship between the gradient of $f$ and the eigendecomposition of $\mC$.

\begin{lemma}
\label{lem:avgsqgrad}
The mean-squared directional derivative of $f$ with respect to the eigenvector $\vw_i$ is equal to the corresponding eigenvalue, $\Exp{((\nabla_{\vx} f)^T\vw_i)^2} \;=\; \lambda_i$.
\end{lemma}

\begin{proof}
By the definition of $\mC$,
\begin{equation}
\lambda_i \;=\; \vw_i^T\mC\vw_i 
\;=\; \vw_i^T\left(\Exp{(\nabla_{\vx} f)\,(\nabla_{\vx} f)^T}\right)\vw_i 
\;=\; \Exp{((\nabla_{\vx} f)^T\vw_i)^2},
\end{equation}
as required.
\end{proof}


The eigenvectors $\mW$ define a rotation of $\mathbb{R}^m$ and consequently the domain of $f$. With eigenvalues in decreasing order, we can separate components of the rotated coordinate system into a set that corresponds to greater average variation and a set corresponding to smaller average variation. The eigenvalues and eigenvectors are partitioned
\begin{equation}
\label{eq:w1}
\Lambda \;=\; \bmat{\Lambda_1 & \\ & \Lambda_2},\qquad
\mW \;=\; \bmat{\mW_1 & \mW_2},
\end{equation}
where $\Lambda_1=\mathrm{diag}\,(\lambda_1,\dots,\lambda_n)$ with $n<m$, and $\mW_1$ is $m \times n$. Define the rotated coordinates $\vy\in\mathbb{R}^n$ and $\vz\in\mathbb{R}^{m-n}$ by
\begin{equation}
\label{eq:yz}
\vy \;=\; \mW_1^T\vx,\quad \vz \;=\; \mW_2^T\vx.
\end{equation}
Then we have the following lemma.
\begin{lemma}
\label{lem:grad}
The mean-squared gradients of $f$ with respect to the coordinates $\vy$ and $\vz$ satisfy
\begin{equation}
\begin{aligned}
\Exp{(\nabla_{\vy} f)^T(\nabla_{\vy} f)} &= \lambda_1+\cdots+\lambda_n, \\
\Exp{(\nabla_{\vz} f)^T(\nabla_{\vz} f)} &= \lambda_{n+1}+\cdots+\lambda_m.
\end{aligned}
\end{equation}
\end{lemma}

\begin{proof}
First note that we can write
\begin{equation}
\label{eq:decomp}
f(\vx) \;=\; f(\mW\mW^T\vx) 
\;=\; f(\mW_1\mW_1^T\vx + \mW_2\mW_2^T\vx) 
\;=\; f(\mW_1\vy + \mW_2\vz).
\end{equation}
By the chain rule, the gradient of $f$ with respect to $\vy$ can be written
\begin{equation}
\label{eq:chainrule}
\nabla_{\vy} f(\vx) \;=\; 
\nabla_{\vy} f(\mW_1\vy + \mW_2\vz) \;=\; 
\mW_1^T\nabla_{\vx}f(\mW_1\vy + \mW_2\vz) \;=\; 
\mW_1^T\nabla_{\vx}f(\vx).
\end{equation}
Then
\begin{equation}
\begin{aligned}
\Exp{(\nabla_{\vy} f)^T(\nabla_{\vy} f)}
&= \Exp{\trace{(\nabla_{\vy} f)(\nabla_{\vy} f)^T}} \\
&= \trace{\Exp{(\nabla_{\vy} f)(\nabla_{\vy} f)^T}} \\
&= \trace{\mW_1^T\Exp{(\nabla_{\vx} f)(\nabla_{\vx} f)^T}\mW_1} \\
&= \trace{\mW_1^T\mC\mW_1} \\
&= \trace{\Lambda_1} \\
&= \lambda_1+\cdots+\lambda_n,
\end{aligned}
\end{equation}
as required. The derivation for the $\vz$ components is similar.
\end{proof}

Lemma \ref{lem:grad} motivates the use of the label \emph{active subspace}. In particular, $f$ varies more on average along the directions defined by the columns of $\mW_1$ than along the directions defined by the columns of $\mW_2$, as quantified by the eigenvalues of $\mC$.  When the eigenvalues $\lambda_{n+1},\dots,\lambda_m$ are all zero, Lemma \ref{lem:grad} implies that the gradient $\nabla_{\vz}f$ is zero everywhere in $\sX$. We call such functions \emph{$\vz$-invariant}.  The next proposition shows that $\vz$-invariant functions have both linear contours and linear isoclines. Similar arguments can be used for higher-order derivatives when they exist.

\begin{proposition}
Let $f$ be $\vz$-invariant, i.e., $\lambda_{n+1}=\cdots=\lambda_m=0$. Then for any two points $\vx_1,\vx_2\in\sX$ such that $\mW_1^T\vx_1=\mW_1^T\vx_2$, $f(\vx_1)=f(\vx_2)$ and $\nabla_{\vx}f(\vx_1) = \nabla_{\vx}f(\vx_2)$.
\end{proposition}

\begin{proof}
The gradient $\nabla_{\vz}f$ being zero everywhere in $\sX$ implies that $f(\vx_1)=f(\vx_2)$. To show that the gradients are equal, assume that $\vx_1$ and $\vx_2$ are on the interior of $\sX$. Then for arbitrary $\vc\in\mathbb{R}^m$, define
\begin{equation}
\vx_1' = \vx_1+\varepsilon\vc,\qquad
\vx_2' = \vx_2+\varepsilon\vc,
\end{equation}
where $\varepsilon>0$ is chosen so that $\vx_1'$ and $\vx_2'$ are in $\sX$. Note that $\mW_1^T\vx_1' = \mW_1^T\vx_2'$ so  $f(\vx_1')=f(\vx_2')$. Then
\begin{equation}
\vc^T\left(\nabla_\vx f(\vx_1)-\nabla_\vx f(\vx_2)\right)
\;=\; \lim_{\varepsilon\rightarrow 0} \frac{1}{\varepsilon} 
\left[
(f(\vx_1')-f(\vx_1))\,-\,(f(\vx_2')-f(\vx_2))
\right] \;=\; 0.
\end{equation}
Simple limiting arguments can be used to extend this result to $\vx_1$ or $\vx_2$ on the boundary of $\sX$.
\end{proof}

\subsection{Two special cases}
We present two cases where the rank of $\mC$ may be determined a priori. The first is a \emph{ridge function}~\cite{Cohen11}, which has the form $f(\vx) \;=\; h(\va^T\vx)$, where $h$ is a univariate function, and $\va$ is a constant $m$-vector. In this case, $\mC$ is rank one, and the eigenvector defining the active subspace is $\va/\|\va\|$, which can be discovered by a single evaluation of the gradient anywhere in $\sX$. The function shown in Figure \ref{fig:expfun} is an example of a ridge function.

The second special case is a function of the form $f(\vx) \;=\;h(\vx^T\mA\vx)$, where $h$ is a univariate function and $\mA$ is a symmetric $m\times m$ matrix. In this case
\begin{equation}
\mC \;=\; 4\,\mA \,\Exp{(h')^2\,\vx\vx^T}\,\mA^T, 
\end{equation}
where $h'=h'(\vx^T\mA\vx)$ is the derivative of $h$. This implies that the null space of $\mC$ is the null space of $\mA$ provided that $h'$ is non-degenerate.

\subsection{Discovering the active subspace}
\label{sec:compdir}
We must compute the eigenvectors $\mW$ and eigenvalues $\Lambda$ of the matrix $\mC$ from \eqref{eq:C}. We immediately encounter the obstacle of computing the elements of $\mC$, which are integrals over the high dimensional space $\sX$. Thus, tensor product numerical quadrature rules are impractical. We opt for Monte Carlo integration, which will yield its own appealing interpretation. In particular, let
\begin{equation}
\nabla_{\vx}f_j \;=\; \nabla_{\vx}f(\vx_j),\qquad \vx_j\in\sX,\quad
j=1,\dots,M,
\end{equation}
be independently computed samples of the gradient vector, where $\vx_j$ is drawn from the density $\rho$ on $\sX$. In practice computing the gradient at $\vx_j$ typically involves first evaluating $f_j = f(\vx_j)$; we will use these function evaluations when testing the response surface.  With the samples of the gradient, we approximate
\begin{equation}
\label{eq:Ctilde}
\mC \;\approx\; \mCt 
\;=\;
\frac{1}{M}\sum_{j=1}^M (\nabla_{\vx}f_j)(\nabla_{\vx}f_j)^T, 
\end{equation}
and compute the eigenvalue decomposition $\mCt \;=\; \mWt\tilde{\Lambda}\mWt^T$.  The size of $\mCt$ is $m\times m$, where we expect $m$ to be on the order of hundreds or thousands corresponding to the number of variables $\vx$. Thus we anticipate no memory limitations when computing the complete eigendecomposition of $\mCt$ on a modern personal computer.

There is another interpretation of the sampling approach to approximate the eigenpairs of $\mC$. We can write $\mCt \;=\; \mG\mG^T$, where the $m\times M$ matrix $\mG$ is
\begin{equation}
\mG \;=\; \frac{1}{\sqrt{M}}\bmat{ \nabla_{\vx}f_1 & \cdots &
  \nabla_{\vx}f_M }.
\end{equation}
If we compute the singular value decomposition (SVD) of $\mG$, then with elementary manipulations,
\begin{equation}
\label{eq:svd}
\mG \;=\; \mWt\sqrt{\tilde{\Lambda}}\mV^T.
\end{equation}
This provides an alternative computational approach via the SVD. Again, we stress that the number of variables $m$ and the number of gradient samples $M$ are small enough in many applications of interest that the SVD can easily be computed on a modern personal
computer.  More importantly, the SVD shows that the rotation matrix $\mWt$ can be interpreted as the uncentered principal directions~\cite{Jolliffe2002} from an ensemble of gradient evaluations. 

It is natural to ask how large $M$ must be for an accurate approximation of the eigenvectors; this is one focus of our current research efforts. If nothing is known a priori about $\mC$, then at least $m$ evaluations are necessary (though maybe not sufficient) to approximate a full rank $\mC$. However, we hypothesize that the number of samples needed for accurate approximation may be related to the rank of $\mC$. Loosely speaking, $f$ must be very smooth for the Monte Carlo approximation to be effective. If $f$'s variability is limited to a small subset of the high-dimensional domain, then the samples of the gradient may not reveal true directions of variability. We are currently exploring how to make such intuitive statements more precise and how to create robust sampling approaches for extreme cases, e.g., a step function in high-dimensions.

In practice, we use the eigenpairs $\mWt$ and $\tilde{\Lambda}$ from the finite-sample approximation $\mCt$ in place of the true eigenpairs $\mW$ and $\Lambda$ of $\mC$ from \eqref{eq:C}. There may be numerical integration methods that produce better approximations than simple Monte Carlo; the latter merely offers an appealing interpretation in terms of the principle components of the gradients. Sequential sampling techniques~\cite{doucet2001} combined with a measure of the stability of the computed subspaces could be a powerful approach for accurate approximation with relatively few samples of the gradient. Alternatively, randomized algorithms for low rank approximation offer promise for reducing the number of gradient samples~\cite{Cai10,Halko2011}.  Quantifying the error in these finite-sample approximations is beyond the scope of this paper. However, in Section \ref{sec:perturbed} we will examine the effects of using the perturbed $\mWt$ to construct the response surface given an estimate of the perturbation.

%% file: sec2-approx.tex
\section{Approximation in the active subspace}
We assume that the number of variables $m$ is too large to permit
standard response surface constructions that suffer from the curse of
dimensionality---such as regression or interpolation.  The goal is to
approximate the $m$-variate function $f$ by a function that is
$\vz$-invariant. If $f$ is nearly $\vz$-invariant, then we expect a
good approximation. A $\vz$-invariant function only varies with
changes in the $n<m$ coordinates $\vy$. Therefore, we can build a
response surface approximation using only the variables $\vy$. Note
that this requires (at least) two levels of approximation: (i)
approximating $f$ by a $\vz$-invariant function and (ii) building a
response surface of the $n$-variate approximation. In this section we
develop the framework and error analysis for this type of
approximation. 

A few preliminaries: define the joint density function $\pi$ of the
coordinates $\vy$ and $\vz$ from \eqref{eq:yz} as
\begin{equation}
\pi(\vy,\vz) \;=\; \rho(\mW_1\vy + \mW_2\vz).
\end{equation}
With this definition, we can define marginal densities $\pi_Y(\vy)$, 
$\pi_Z(\vz)$ and conditional densities $\pi_{Y|Z}(\vy|\vz)$, 
$\pi_{Z|Y}(\vz|\vy)$ in the standard way.
Next we define the domain of a function that only depends on
$\vy$. Define the set $\sY$ to be
\begin{equation}
\label{eq:Y}
\sY \;=\;
\left\{\,\vy\,:\,\vy=\mW_1^T\vx,\,\vx\in\sX\,\right\} 
\;\subseteq\; \mathbb{R}^n.
\end{equation}
Note that the marginal density $\pi_Y(\vy)$ defines a probability
density on $\sY$. With these defined we can begin approximating.

\subsection{Conditional expectation}
For a fixed $\vy$, the best guess one can make at the value of $f$ is
its average over all values of $\vx$ that map to $\vy$; this is
precisely the conditional expectation of $f$ given $\vy$.  Define the
function $G$ that depends on $\vy$ by
\begin{equation}
\label{eq:condexp}
G(\vy) \;=\; \CExp{f}{y}
\;=\; \int_{\vz} f(\mW_1\vy+\mW_2\vz)\,\pi_{Z|Y}(\vz)\,d\vz.
\end{equation}
The second equality follows from the so-called \emph{law of the
  unconscious statistician}. The domain of this
function is $\sY$ from \eqref{eq:Y}. Since $G$ is a conditional
expectation, it is the best mean-squared approximation of $f$ given
$\vy$~\cite[Chapter 9]{Williams1991}.

We can use $G$ to approximate $f$ at a given $\vx$ with the following 
construction,
\begin{equation}
\label{eq:lowdapprox}
f(\vx) \;\approx\; F(\vx) \;\equiv\; G(\mW_1^T\vx).
\end{equation}
The next theorem provides an error bound for $F$ in terms of the 
eigenvalues of $\mC$ from \eqref{eq:C}.

\begin{theorem}
\label{thm:error}
The mean squared error of $F$ defined in \eqref{eq:lowdapprox} satisfies
\begin{equation}
\label{eq:Fbound}
\Exp{(f-F)^2} \;\leq\; C_1\,(\lambda_{n+1}+\cdots+\lambda_m)
\end{equation}
where $C_1$ is a constant that depends only on the domain $\sX$ and
the weight function $\rho$.
\end{theorem}

\begin{proof}
Note that $\CExp{f-F}{\vy}=0$ by the definition \eqref{eq:lowdapprox}. Thus,
\begin{align}
\Exp{(f-F)^2} &= \Exp{\CExp{(f-F)^2}{\vy}} \label{line1}\\
&\leq C_1\, \Exp{\CExp{(\nabla_{\vz}f)^T(\nabla_{\vz}f)}{\vy}} \label{line2}\\
&= C_1 \,\Exp{(\nabla_{\vz}f)^T(\nabla_{\vz}f)} \label{line3}\\
&= C_1\,(\lambda_{n+1} + \cdots + \lambda_m).\label{line4}
\end{align}
Lines \eqref{line1} and \eqref{line3} are due to the tower property of
conditional expectations. Line \eqref{line2} is a Poincar\'{e} inequality,
where the constant $C_1$ depends only on $\sX$ and the density
function $\rho$. Line \eqref{line4} follows from Lemma \ref{lem:grad}.
\end{proof}

\subsection{Monte Carlo approximation}
\label{sec:mcapprox}
The trouble with the approximation $F$ from \eqref{eq:lowdapprox} is
that each evaluation of $F$ requires an integral with respect to the
$\vz$ coordinates. In other words, evaluating $F$ requires
high-dimensional integration. However, if $f$ is nearly
$\vz$-invariant, then it is nearly constant along the coordinates
$\vz$. Thus, its variance along $\vz$ will be very small, and we
expect that simple numerical integration schemes to approximate the
conditional expectation $G$ will work well. We use simple Monte Carlo
to approximate $G$ and derive an error bound on such an
approximation. The error bound validates the intuition that we need
very few evaluations of $f$ to approximate $G$ if $f$ is nearly
$\vz$-invariant.

Define the Monte Carlo estimate $\hat{G}=\hat{G}(\vy)$ by
\begin{equation}
\label{eq:condexpmc}
G(\vy) \;\approx\; \hat{G}(\vy) \;=\; \frac{1}{N}\sum_{i=1}^N
f(\mW_1\vy + \mW_2\vz_i),
\end{equation}
where the $\vz_i$ are drawn independently from the conditional density
$\pi_{Z|Y}$. We approximate $f$ as
\begin{equation}
\label{eq:lowdapproxmc}
f(\vx) \;\approx\; \hat{F}(\vx) \;\equiv\;
\hat{G}(\mW_1^T\vx).
\end{equation}
Next we derive an error bound for this approximation.

\begin{theorem}
\label{thm:errormc}
The mean squared error of $\hat{F}$ defined in \eqref{eq:lowdapproxmc}
satisfies
\begin{equation}
\label{eq:Fhatbound}
\Exp{(f-\hat{F})^2} \;\leq\; C_1\,\left(1+\frac{1}{N}\right)\,
(\lambda_{n+1}+\cdots+\lambda_m)
\end{equation}
where $C_1$ is from Theorem \ref{thm:error}.
\end{theorem}

\begin{proof}
First define the conditional variance of $f$ given $\vy$ as
$\sigma^2_{\vy} \;=\; \CExp{(f-F)^2}{\vy}$,
and note that the proof of Theorem \ref{thm:error} shows
\begin{equation}
\Exp{\sigma^2_{\vy}}
\;\leq\;C_1\,\left(\lambda_{n+1}+\cdots+\lambda_m\right).
\end{equation}
Next note that the mean-squared error in the Monte Carlo approximation
satisfies~\cite{owenmc-2013},
\begin{equation}
\label{eq:cexperror}
\CExp{(F-\hat{F})^2}{\vy} \;=\; \frac{\sigma^2_{\vy}}{N},
\end{equation}
so that
\begin{equation}
\Exp{(F-\hat{F})^2} \;=\; \Exp{\CExp{(F-\hat{F})^2}{\vy}}
\;=\; \frac{1}{N}\,\Exp{\sigma^2_{\vy}}
\;\leq\; \frac{C_1}{N}\,\left(\lambda_{n+1}+\cdots+\lambda_m\right).
\end{equation}
Finally, using Theorem \ref{thm:error},
\begin{equation}
\begin{aligned}
\Exp{(f-\hat{F})^2} &\leq \Exp{(f-F)^2}\,+\,\Exp{(F-\hat{F})^2}\\
&\leq C_1\,\left(1+\frac{1}{N}\right)\,
\left(\lambda_{n+1}+\cdots +\lambda_m\right),
\end{aligned} 
\end{equation}
as required.
\end{proof}

This bound shows that if $\lambda_{n+1},\dots,\lambda_m$ are
sufficiently small, then the Monte Carlo estimate with small $N$
(e.g., $N=1$) will produce a very good approximation of $f$.

\subsection{Response surfaces}
We now reach the point where $n<m$ can reduce the cost of
approximating $f$. Up to this point, there has been no advantage to
using the conditional expectation $F$ or its Monte Carlo approximation
$\hat{F}$ to approximate $f$; each evaluation of $\hat{F}(\vx)$
requires at least one evaluation of $f(\vx)$. The real advantage of
this method is that one can construct response surfaces with respect
to the few variables $\vy\in\mathbb{R}^n$ instead of $f$'s natural
variables $\vx\in\mathbb{R}^m$. We will train a
response surface on the domain $\sY\subseteq\mathbb{R}^n$ using a set
of evaluations of $\hat{G}=\hat{G}(\vy)$.

Here we do not specify the form of the response surface; several are
possible, and Section \ref{sec:kriging} discusses applications using
kriging. However, there is one important consideration before choosing
a response surface method willy-nilly. If the eigenvalues
$\lambda_{n+1},\dots,\lambda_m$ are not exactly zero, then evaluations
of the Monte Carlo approximation $\hat{G}$ will contain \emph{noise}
due to the finite number of samples; in other words, the Monte Carlo
estimate $\hat{G}$ is a random variable. This noise implies that
$\hat{G}$ is not a smooth function of $\vy$. Thus, we prefer smoothing
regression-based response surfaces over exact interpolation.
In Section \ref{sec:kriging}, we characterize the
noise and use it to tune the parameters of a kriging surface.

We construct a generic response surface for a function defined on
$\sY$ from \eqref{eq:Y} as follows. Define the \emph{design} on $\sY$
to be a set of points $\vy_k\in\sY$ with $k=1,\dots,P$. The specific
design will depend on the form of the response surface. Define $\hat{G}_k=\hat{G}(\vy_k)$. Then we approximate
\begin{equation}
\label{eq:respsurf}
\hat{G}(\vy) \;\approx\;
\tilde{G}(\vy) \;\equiv\;
\sR(\vy;\,\hat{G}_1,\dots,\hat{G}_P),
\end{equation}
where $\sR$ is a response surface constructed with the training data
$\hat{G}_1,\dots,\hat{G}_P$.  We use this response surface to
approximate $f$ as
\begin{equation}
\label{eq:Ftilde}
f(\vx) \;\approx\; \tilde{F}(\vx) \;\equiv\; \tilde{G}(\mW_1^T\vx).
\end{equation}
To derive an error estimate for $\tilde{F}$, we assume the 
error in the response surface can be bounded as follows.

\begin{assumption}
\label{bigass}
Let $\sZ=\{\vz\,:\,\vz=\mW_2^T\vx,\;\vx\in\sX\}$. Then there exists a
constant $C_2$ such that
\begin{equation}
\CExp{(\hat{F}-\tilde{F})^2}{\vz}\;\leq\; C_2\delta
\end{equation}
for all $\vz\in\sZ$, where $\delta=\delta(\sR,P)$ depends on the 
response surface method $\sR$ and the number $P$ of training data, 
and $C_2$ depends on the domain $\sX$ and the probability density
function $\rho$. 
\end{assumption}

With Assumption \ref{bigass}, we have the
following error estimate for $\tilde{F}$.

\begin{theorem}
\label{thm:errorrs}
The mean-squared error in $\tilde{F}$ defined in \eqref{eq:Ftilde} satisfies
 \begin{equation}
\Exp{(f-\tilde{F})^2} \;\leq\; C_1\left(1+\frac{1}{N}\right)\,
\left(\lambda_{n+1}+\cdots+\lambda_{m}\right) \; +\; C_2\,\delta,
\end{equation}
where $C_1$ is from Theorem \ref{thm:error}, $N$ is from Theorem
\ref{thm:errormc}, and $C_2$ and $\delta$ are from Assumption 
\ref{bigass}.
\end{theorem}

\begin{proof}
Note
\begin{equation}
\Exp{(f-\tilde{F})^2} \;\leq\; \Exp{(f-\hat{F})^2}\,+\,
\Exp{(\hat{F}-\tilde{F})^2}.
\end{equation}
Theorem \ref{thm:errormc} bounds the first summand. By the tower
property and Assumption \ref{bigass}, the second
summand satisfies
\begin{equation}
\Exp{(\hat{F}-\tilde{F})^2} 
\;=\; \Exp{\CExp{(\hat{F}-\tilde{F})^2}{\vz}}
\;\leq\; C_2\delta,
\end{equation}
as required.
\end{proof}

The next equation summarizes the three levels of approximation:
\begin{equation}
\begin{array}{ccccccc}
f(\vx)&\approx&F(\vx)&\approx&\hat{F}(\vx)&\approx&\tilde{F}(\vx) \\
 & &\verteq& &\verteq& &\verteq \\
 & &G(\mW_1^T\vx)&\approx&\hat{G}(\mW_1^T\vx)&\approx&
\tilde{G}(\mW_1^T\vx) \\
\end{array}
\end{equation}
The conditional expectation $G$ is defined in \eqref{eq:condexp};
its Monte Carlo approximation is $\hat{G}$ is defined in
\eqref{eq:condexpmc}; and the response surface $\tilde{G}$
is defined in \eqref{eq:respsurf}. The respective error
estimates are given in Theorems \ref{thm:error}, \ref{thm:errormc},
and \ref{thm:errorrs}.

\subsection{Using perturbed directions}
\label{sec:perturbed}
Up to this point, we have assumed that we have the exact eigenvectors
$\mW$. However, as discussed in Section
\ref{sec:compdir}, in practice we only have an perturbed version
$\mWt$---although both the true $\mW$ and the perturbed $\mWt$ are
orthonormal. In this section we examine the effects of this
perturbation on the approximation of $f$.
We assume the following characterization of the perturbation.

\begin{assumption}
\label{bigass2}
Given $\mW$ from \eqref{eq:gradcovariance}, let $\mWt$ be a perturbed
version of $\mW$ that satisfies the following two conditions: (i) the
sign of $\vwt_i$, the $i$th column of $\mWt$, is chosen to minimize
$\|\vw_i-\vwt_i\|$, (ii) there is an $\varepsilon>0$ such that the
perturbation satisfies $\|\mW-\mWt\|\leq\varepsilon$ in the matrix
2-norm.
\end{assumption}

\begin{lemma}
\label{lem:ortho}
Given the partition of $\mW$ and a comparable partition
$\mWt=\bmat{\mWt_1&\mWt_2}$, $\|\mW_2^T\mWt_2\|\leq 1$ and
$\|\mW_1^T\mWt_2\|\leq\varepsilon$ in the matrix 2-norm.
\end{lemma}

\begin{proof}
The orthogonality of the columns of $\mW_2$ and $\mWt_2$ implies
$\|\mW_2^T\mWt_2\|\leq\|\mW_2^T\|\|\mWt_2\|=1$. The second inequality
follows from
\begin{equation}
\|\mW_1^T\mWt_2-\mathbf{0}\| \;=\; \|\mW_1^T(\mWt_2-\mW_2) \|
\;=\; \|\mWt_2-\mW_2\| \;\leq\;\varepsilon
\end{equation}
where the last relation follows from Assumption \ref{bigass2}.
\end{proof}

The perturbed $\mWt$ define perturbed coordinates $\vyt=\mWt_1^T\vx$
and $\vzt=\mWt_2^T\vx$. The joint density
$\tilde{\pi}(\vyt,\vzt)=\rho(\mWt_1\vyt + \mWt_2\vzt)$ begets marginal
densities $\tilde{\pi}_{\tilde{Y}}(\vyt)$,
$\tilde{\pi}_{\tilde{Z}}(\vzt)$ and conditional densities
$\tilde{\pi}_{\tilde{Y}|\tilde{Z}}(\vyt|\vzt)$,
$\tilde{\pi}_{\tilde{Z}|\tilde{Y}}(\vzt|\vyt)$. The domain of
the perturbed approximations is 
\begin{equation}
\sYt \;=\; \left\{\,\vyt\,:\, \vyt=\mWt_1^T\vx,\,\vx\in\sX\,\right\}.
\end{equation}
We construct the same
sequence of approximations of $f$ using these perturbed
coordinates. We denote the perturbed versions of the approximations
with a subscript $\varepsilon$.  The conditional expectation
approximation of $f$ becomes $F_{\varepsilon}(\vx)\equiv
G_{\varepsilon}(\mWt_1^T\vx)$, where
\begin{equation}
\label{eq:condexppert}
G_{\varepsilon}(\vyt) \;=\; \CExp{f}{\vyt} \;=\; \int_{\vzt}
f(\mWt_1\vyt +
\mWt_2\vzt)
\,\tilde{\pi}_{\tilde{Z}|\tilde{Y}}(\vzt|\vyt)\,d\,\vzt.
\end{equation}
Then we have the following error estimate

\begin{theorem}
\label{thm:errorpert}
The mean-squared error in the conditional expectation $F_\varepsilon$
using the perturbed eigenvectors $\mWt$ satisfies
\begin{equation}
\Exp{(f-F_\varepsilon)^2} \;\leq\; C_1\left(\, 
\varepsilon\left(\lambda_1+\cdots+\lambda_n\right)^{\frac{1}{2}}
\;+\;\left(\lambda_{n+1}+\cdots+\lambda_m\right)^{\frac{1}{2}}
\,\right)^2
\end{equation}
where $C_1$ is from Theorem
\ref{thm:error}.
\end{theorem}

\begin{proof}
Following the same reasoning as the proof of Theorem \ref{thm:error} 
using the Poincar\'{e} inequality,
\begin{equation}
\Exp{(f-F_\varepsilon)^2} \;\leq\; C_1\Exp{(\nabla_{\vzt}f)^T(\nabla_{\vzt}f)}.
\end{equation}
Using the chain rule, $\nabla_{\vzt}f \;=\; \mW_2^T\mWt_2\nabla_{\vz}f
+ \mW_1^T\mWt_2\nabla_{\vy}f$.
Then 
\begin{align}
\Exp{(\nabla_{\vzt}f)^T(\nabla_{\vzt}f)} 
&\leq
\Exp{(\nabla_\vz f)^T(\nabla_\vz f)} +
2\varepsilon\,\Exp{(\nabla_\vz f)^T(\nabla_\vy f)} +
\varepsilon^2\Exp{(\nabla_\vy f)^T(\nabla_\vy f)} \label{eq:usebigass}\\
&\leq
\left(
\Exp{(\nabla_\vz f)^T(\nabla_\vz f)}^{\frac{1}{2}}
+ \varepsilon\,\Exp{(\nabla_\vy f)^T(\nabla_\vy f)}^{\frac{1}{2}}
\right)^2 \label{eq:cauchyschwarz}\\
&\leq
\left(\, 
\varepsilon\left(\lambda_1+\cdots+\lambda_n\right)^{\frac{1}{2}}
\;+\;\left(\lambda_{n+1}+\cdots+\lambda_m\right)^{\frac{1}{2}}
\,\right)^2. \label{eq:uselemma1}
\end{align}
Line \eqref{eq:usebigass} follows from Lemma \ref{lem:ortho}. Line \eqref{eq:cauchyschwarz} follows from the Cauchy-Schwarz inequality. Line
\eqref{eq:uselemma1} follows from Lemma \ref{lem:grad}
\end{proof}

Notice how the eigenvalues $\lambda_1,\dots,\lambda_n$ contribute
to the error estimate given the perturbation bound $\varepsilon$
in the eigenvectors. This contribution persists in error estimates
for the Monte Carlo approximation and the response surface using
the perturbed eigenvectors.

Let $\hat{F}_\varepsilon(\vx)\equiv\hat{G}_\varepsilon(\mWt^T\vx)$ where
\begin{equation}
\label{eq:condexpmcpert}
\hat{G}_\varepsilon(\vyt) \;=\; 
\frac{1}{N}\sum_{i=1}^N f(\mWt_1\vyt + \mWt_2\vzt_i),
\end{equation}
where $\vzt_i$ are drawn from the conditional density
$\tilde{\pi}_{\tilde{Z}|\tilde{Y}}(\vzt|\vyt)$. Then we have the
following error estimate, whose derivation follows the proof of
Theorem \ref{thm:errormc} using the perturbed coordinates and the
reasoning from the proof of Theorem \ref{thm:errorpert}.

\begin{theorem}
\label{thm:errormcpert}
The mean-squared error in the Monte Carlo approximation
$\hat{F}_\varepsilon$ using the perturbed eigenvectors $\mWt_1$
satisfies
\begin{equation}
\Exp{(f-\hat{F}_\varepsilon)^2}
\;\leq\; C_1
\left(1+\frac{1}{N}\right)
\left(\, 
\varepsilon\left(\lambda_1+\cdots+\lambda_n\right)^{\frac{1}{2}}
\;+\;\left(\lambda_{n+1}+\cdots+\lambda_m\right)^{\frac{1}{2}}
\,\right)^2
\end{equation}
where $C_1$ and $N$ are the quantities from Theorem
\ref{thm:errormc}.
\end{theorem}

The response surface approximation using the perturbed eigenvectors
is $\tilde{F}_\varepsilon(\vx)\equiv \tilde{G}_\varepsilon(\mWt_1^T\vx)$,
where
\begin{equation}
\label{eq:respsurfpert}
\tilde{G}_\varepsilon(\vyt) \;=\; \sR(\vyt;\,\hat{G}_{\varepsilon,1},
\dots,\hat{G}_{\varepsilon,P})
\end{equation}
for a chosen response surface method $\sR$.  The
$\hat{G}_{\varepsilon,k}$ are evaluations of $\hat{G}_\varepsilon$ at
the design points $\vyt_k\in\sYt$. We have the following error
estimate; again, its derivation follows the proof of Theorem
\ref{thm:errorrs} using the reasoning from the proof of Theorem
\ref{thm:errorpert}.

\begin{theorem}
\label{thm:errorrspert}
Under the assumptions of Theorem \ref{thm:errorrs}, the mean-squared
error in the response surface approximation $\tilde{F}_\varepsilon$
satisfies
\begin{equation}
\Exp{(f-\tilde{F}_\varepsilon)^2} 
\;\leq\; C_1\left(1+\frac{1}{N}\right)\,
\left(\, 
\varepsilon\left(\lambda_1+\cdots+\lambda_n\right)^{\frac{1}{2}}
\;+\;\left(\lambda_{n+1}+\cdots+\lambda_m\right)^{\frac{1}{2}}
\,\right)^2
\; +\; C_2\,\delta,
\end{equation}
where $C_1$, $N$, $C_2$, and $\delta$ are the quantities
from Theorem \ref{thm:errorrs}.
\end{theorem}

We summarize the three levels of approximation using
the perturbed eigenvectors as
\begin{equation}
\begin{array}{ccccccc}
f(\vx)&\approx&F_\varepsilon(\vx)&\approx&\hat{F}_\varepsilon(\vx)
&\approx&\tilde{F}_\varepsilon(\vx) \\
 & &\verteq& &\verteq& &\verteq \\
 & &G_\varepsilon(\mWt_1^T\vx)&\approx&\hat{G}_\varepsilon(\mWt_1^T\vx)
&\approx&\tilde{G}_\varepsilon(\mWt_1^T\vx) \\
\end{array}
\end{equation}
The conditional expectation $G_\varepsilon$ is defined in
\eqref{eq:condexppert}; its Monte Carlo approximation is
$\hat{G}_\varepsilon$ is defined in \eqref{eq:condexpmcpert}; and the
response surface $\tilde{G}_\varepsilon$ is defined in
\eqref{eq:respsurfpert}. The respective error estimates are given in
Theorems \ref{thm:errorpert}, \ref{thm:errormcpert}, and \ref{thm:errorrspert}.

%% file: sec3-kriging.tex
\section{Heuristics for kriging surfaces}
\label{sec:kriging}

In this section, we detail a heuristic procedure to construct a kriging surface~\cite{Koehler1996} (also known as
Gaussian process approximation~\cite{Rasmussen2006} and closely related to
radial basis approximation~\cite{Wendland2005}) on the $n$-dimensional reduced domain $\sY$ from \eqref{eq:Y} defined by the left singular vectors $\mWt$ from the samples of the gradient of $f$. To keep the notation clean, we use $\mW$ instead of $\mWt$, and we do not explore the effects of the perturbed directions for this particular heuristic. 

We must first choose the dimension $n$ of the subspace. Many covariance-based
reduction methods (e.g., the proper orthogonal decomposition~\cite{Sirovich87a}) use the magnitude of the $\lambda_i$ to define $n$, e.g., so that $\lambda_1+\cdots+\lambda_n$ exceeds some proportion of $\lambda_1+\cdots+\lambda_m$. We are bound instead by more practical considerations, such as choosing $n$ small enough to construct a reasonable design (e.g., a mesh) for the kriging surface. The trailing eigenvalues $\lambda_{n+1},\dots,\lambda_m$ then inform a noise model as discussed in Section \ref{sec:training}. A rapid decay in the $\lambda_i$ implies that the low-dimensional approximation is relatively less noisy. 

\subsection{Design on reduced domain}
\label{sec:design}
We need to choose the points $\vy_k$ on the reduced domain $\sY$ where we evaluate $\hat{G}$ and construct the kriging surface. We restrict our attention to $\sY$ derived from two particular choices of the domain $\sX$ and the density function $\rho$ that are often found in practice: a Gaussian density on $\mathbb{R}^m$ and a uniform density on a hypercube. 

\subsubsection{The Gaussian case}
The first case we consider is when the domain $\sX$ is $\mathbb{R}^m$ and $\rho(\vx)$ is a Gaussian density with mean zero and an identity covariance. In this case, the reduced domain $\sY$ is $\mathbb{R}^n$, and the marginal density $\pi_Y(\vy)$ is also a zero-mean Gaussian with an identity covariance, since $\vy=\mW_1^T\vx$ and $\mW_1^T\mW_1=\mI$. We choose a simple tensor product design (i.e., a grid or lattice) on $\sY$ such that each univariate design covers three standard deviations. For example, a nine-point design in $\mathbb{R}^2$ would use the points $\{-3,0,3\}\times\{-3,0,3\}$. This is the approach we will take in the numerical experiments in Section \ref{sec:example}.

\subsubsection{The uniform case}
Next, assume that $\sX = [-1,1]^m$, which we write equivalently as $-1\leq \vx \leq 1$, and the density $\rho$ is uniform over $[-1,1]^m$. The reduced domain becomes
\begin{equation}
\sY \;=\; \left\{\,\vy\,:\, \vy=\mW_1^T\vx,\; -1\leq \vx\leq 1\,\right\}.
\end{equation}
In general, $\sY$ will not be a hypercube in $\mathbb{R}^n$---only if $\mW_1$ contains only columns of the identity matrix. However, $\sY$ will be a convex polytope in $\mathbb{R}^n$ whose vertices are a subset of vertices of $[-1,1]^m$ projected to $\mathbb{R}^n$. For example, if $m=3$ and $n=2$, then one can imagine taking a photograph of a rotated cube; six of the cube's eight vertices define the polytope in $\mathbb{R}^2$. These projections of hypercubes are called \emph{zonotopes}, and there exist polynomial time algorithms for discovering the vertices that define the convex hull~\cite{fukuda2004zonotope}. In principle, one can give these vertices to a mesh generator (e.g., \cite{persson2004simple}) to create a design on $\sY$. 
The marginal density $\pi_Y(\vy)$ is challenging to compute; our current research efforts include this task. 

\subsection{Training the kriging surface}
\label{sec:training}

In Section \ref{sec:mcapprox}, we described approximating the conditional expectation $G(\vy)$ by its Monte Carlo estimate $\hat{G}(\vy)$ from \eqref{eq:condexpmc}. The error bound in Theorem \ref{thm:errormc} shows that if the trailing eigenvalues $\lambda_{n+1},\dots,\lambda_m$ are small enough, then the number $N$ of samples needed for the Monte Carlo estimate can be very small---even $N=1$. However, the proof of the theorem used the error measure \eqref{eq:cexperror}, which assumes that $\vz_i$ are drawn independently from the conditional density $\pi_{Z|Y}$. In practice, we can use a Metropolis-Hasting~\cite{chib1995} scheme to sample from $\pi_{Z|Y}$ since it is proportional to the given $\rho(\mW_1\vy+\mW_2\vz)$. But these samples are correlated, and the error bound in \eqref{eq:cexperror} does not strictly apply~\cite{chan1994discussion}. 

In practice, we have had success using only a single evaluation of $f$ in the computation of $\hat{G}$, as suggested by Theorem \ref{thm:errormc}. Therefore, we do not need to draw several samples from $\pi_{Z|Y}$; we need only a single $\vx_k\in\sX$ such that $\mW_1^T\vx_k=\vy_k$ for each $\vy_k$ in the design on $\sY$. We can easily find such an $\vx_k$ for each of the two cases discussed in the previous section.
\begin{itemize}
\item If $\sX=\mathbb{R}^m$ and $\rho$ is a Gaussian density, then $\vx_k=\mW_1\vy_k$. 
\item If $\sX=[-1,1]^m$ and $\rho$ is a uniform density, then $\vx_k$ must satisfy $\mW_1^T\vx_k=\vy_k$ and $-1\leq \vx_k\leq 1$. Thus, $\vx_k$ can be found using Phase 1 of a linear program~\cite{YeLP08}. 
\end{itemize}
With $\vx_k$, we compute $\hat{G}_k=\hat{G}(\vy_k)=f(\vx_k)$; the set $\{(\vy_k, \hat{G}_k)\}$ comprise the training data for the kriging surface. 

We assume that the function we are trying to approximate is smooth with respect to the
coordinates $\vy$. However, since the training data $\hat{G}_k$ are not exactly equal to the conditional expectation $G(\vy_k)$, we do not want to force the kriging surface to interpolate the training data. Instead, we want to build a model for the noise in the training data that is motivated by Theorem \ref{thm:error} and incorporate it into the kriging surface. We choose the correlation matrix of the training data to come from a product-type squared exponential kernel with an additional diagonal term to represent the noise,
\begin{equation}
\label{eq:covmat}
\Cov{\tilde{G}(\vy_{k_1})}{\tilde{G}(\vy_{k_2})}
\;=\;
K\left(\vy_{k_1},\vy_{k_2}\right)
+
\eta^2\,\delta(k_1,k_2),
\end{equation}
where $\delta(k_1,k_2)$ is 1 if $k_1=k_2$ and zero otherwise, and
\begin{equation}
\label{eq:kkernel}
K\left(\vy_{k_1},\vy_{k_2}\right)
\;=\;
\exp\left(-\sum_{i=1}^n \frac{(y_{k_1,i} - y_{k_2,i})^2}{2\ell_i^2}\right).
\end{equation}
Along with the correlation function, we choose a quadratic mean term, which will add ${n+2 \choose n}$ polynomial basis functions to the kriging approximation. Note that this imposes a restriction that the design must be poised for quadratic approximation. 

We are left to determine the parameters of the kriging surface, including the
correlation lengths $\ell_i$ from \eqref{eq:kkernel} and the parameter $\eta^2$
of the noise model from \eqref{eq:covmat}. Given values for these parameters,
the coefficients of both the polynomial bases and the linear combination of the
training data are computed in the standard way~\cite{Koehler1996,Rasmussen2006}. We could use a standard maximum likelihood method to compute $\ell_i$ and $\eta^2$. However, we can inform these parameters using the quantities from Lemma \ref{lem:avgsqgrad} and Theorem \ref{thm:error}. 

Toward this end, we approximate the directional derivative of $f$ along $\vw_i$ with a finite difference,
\begin{equation}
\nabla_{\vx}f(\vx)^T\vw_i 
\;\approx\;
\frac{1}{\delta}\left(
f(\vx+\delta\vw_i) - f(\vx)
\right),
\end{equation}
which is valid for small $\delta$. Decompose  $f(\vx) = f_0 + f'(\vx)$, where
$f_0=\Exp{f}$, so that $f'$ has zero-mean. By Lemma \ref{lem:avgsqgrad},
\begin{equation}
\begin{aligned}
\lambda_i &= \Exp{((\nabla_{\vx} f)^T\vw_i)^2} \\
&\approx 
\frac{1}{\delta^2} \Exp{\left( f(\vx+\delta\vw_i) - f(\vx)\right)^2}\\ 
&= 
\frac{1}{\delta^2} \Exp{\left( f'(\vx+\delta\vw_i) - f'(\vx)\right)^2}\\ 
&= 
\frac{2}{\delta^2}\left(
\sigma^2 - \Cov{f'(\vx+\delta\vw_i)}{f'(\vx)}
\right),
\end{aligned}
\end{equation}
where $\sigma^2 = \Var{f}$. Rearranged, we have
\begin{equation}
\Corr{f'(\vx+\delta\vw_i)}{f'(\vx)}
\;=\;
\frac{1}{\sigma^2}\Cov{f'(\vx+\delta\vw_i)}{f'(\vx)}
\;=\;
1 - \frac{\lambda_i\delta^2}{2\sigma^2}.
\end{equation}
This implies that the correlation function along $\vw_i$ is locally quadratic near the origin with coefficient $\lambda_i/2\sigma^2$. A univariate squared exponential correlation function with correlation length parameter $\ell$ has a Taylor series about the origin
\begin{equation}
\exp\left(-\frac{\delta^2}{2\ell^2}\right)
\;=\; 
1 - \frac{\delta^2}{2\ell^2} + \cdots
\end{equation}
Comparing these terms to the locally quadratic approximation of the correlation function, we can use a univariate Gaussian with correlation length
\begin{equation}
\ell_i^2 \;=\; \frac{\sigma^2}{\lambda_i} 
\end{equation}
for the reduced coordinate $y_i$. In other words, the correlation length parameter $\ell_i$ corresponding to $y_i$ in the correlation kernel will be inversely proportional to the square root of the eigenvalue $\lambda_i$. Thus we need to approximate the variance $\sigma^2$. 

Applying Theorem \ref{thm:error} with $n=0$, we have
\begin{equation}
\sigma^2 \;=\; \Var{f} 
\;\leq\; C_1\,(\lambda_1+\cdots+\lambda_m).
\end{equation}
Unfortunately, the Poincar\'{e} inequality is a notoriously loose bound, so we are reluctant to simply plug in an estimate of $C_1$ to approximate $\sigma^2$. Instead, we posit that for some constant $\alpha$,
\begin{equation}
\label{eq:varest}
\sigma^2 \;=\; \alpha\,(\lambda_1+\cdots+\lambda_m),
\end{equation}
where $\alpha\leq C_1$, and we can employ an estimate of $C_1$. For the case when $\sX=\mathbb{R}^m$ and $\rho$ is a standard normal, $C_1\leq 1$~\cite{chen1982inequality}. If $\sX=[-1,1]^m$ and $\rho$ is uniform, then $C_1\leq 2\sqrt{m}/\pi$~\cite{Bebendorf2003}. Other cases can use comparable estimates, if available. 
To get a rough lower bound on $\alpha$, we use the biased estimator of $\sigma^2$ from the samples $f_j$ computed in Section \ref{sec:compdir},
\begin{equation}
\hat{\sigma}^2 \;=\; \frac{1}{M} \sum_{j=1}^M (f_j - \hat{f}_0)^2,
\end{equation}
where $\hat{f}_0$ is the empirical mean of the $f_j$. Since $M$ will generally be small due to limited function evaluations, we expect the bias to be significant enough to justify the bound
\begin{equation}
\label{eq:lowbnd}
\hat{\sigma}^2 \;\leq\; \alpha\,(\lambda_1+\cdots+\lambda_m),
\end{equation}
so that
\begin{equation}
\label{eq:alphabnds}
\frac{\hat{\sigma}^2}{\lambda_1+\cdots+\lambda_m}
\;\leq\;
\alpha
\;\leq\;
C_1.
\end{equation}
From here, we treat $\alpha$ as a hyperparameter for the correlation kernel \eqref{eq:kkernel}, and we can use a maximum likelihood approach to set it. Note that the number of variables in the maximum likelihood is one, in contrast to $n+1$ variables for the standard approach. Given a value for $\alpha$ that achieves the maximum likelihood, we set $\sigma^2$ by \eqref{eq:varest} and 
\begin{equation}
\label{eq:etaval}
\eta^2 \;=\; \alpha\,(\lambda_{n+1}+\cdots+\lambda_m)
\end{equation}
in \eqref{eq:covmat}. We now have all necessary quantities to build the kriging
surface on the active subspace.

\subsection{A step-by-step algorithm}
\label{sec:steps}
We summarize this section with an algorithm incorporating the previously described computational procedures given a function $f=f(\vx)$ and its gradient $\nabla_{\vx}f=\nabla_{\vx}f(\vx)$ defined on $\sX$ with probability density function $\rho$. Portions of this algorithm are specific to the Gaussian and uniform density cases discussed in Section \ref{sec:design}.
\begin{enumerate}
\item \textbf{Initial sampling:} Choose a set of $M$ points $\vx_j\in\sX$
  according to the measure $\rho(\vx)$. For each $\vx_j$, compute $f_j=f(\vx_j)$
  and $\nabla_{\vx}f_j = \nabla_{\vx}f(\vx_j)$. Compute the sample variance
  $\hat{\sigma}^2$. 
\item \textbf{Gradient analysis:} Compute the SVD of the matrix
  \begin{equation}
    \mG \;=\; \frac{1}{\sqrt{M}}\bmat{\nabla_{\vx}f_1 & \cdots &
      \nabla_{\vx}f_M} \;=\; \mW\Sigma\mV^T, 
  \end{equation}
  and set $\Lambda=\Sigma^2$. Choose a reduced dimension $n<m$ according to
  practical considerations and the decay of $\lambda_i$. Partition
  $\mW=\bmat{\mW_1 &\mW_2}$.
\item \textbf{Reduced domain:} If $\sX=\mathbb{R}^m$ and $\rho$ is a standard Gaussian density, then the reduced domain $\sY=\mathbb{R}^n$. If $\sX=[-1,1]^m$ and $\rho$ is a uniform density, then use the method described in~\cite{fukuda2004zonotope} to determine the vertices of the zonotope in $\mathbb{R}^n$. 
\item \textbf{Design on reduced domain:} Choose a set of $P$ points
  $\vy_k\in\sY$. In the Gaussian case, choose $\vx_k=\mW_1\vy_k$. In the uniform case, use a linear program solver to find an $\vx_k\in[-1,1]^m$ that satisfies $\vy_k=\mW_1^T\vx_k$. 
\item \textbf{Train the response surface:}  For each $\vy_k$, set $\hat{G}_k=\hat{G}(\vy_k)=f(\vx_k)$. Use a maximum likelihood method to find the hyperparameter $\alpha$ with bounds from \eqref{eq:alphabnds}.
Set the noise model parameter $\eta^2$ as in \eqref{eq:etaval} and correlation lengths
  \begin{equation}
    \ell_i^2 \;=\; \frac{\alpha}{\lambda_i}\,(\lambda_{1}+\cdots+\lambda_m)
  \end{equation}
for the product squared exponential correlation kernel on $\bar{\sY}$. Apply standard kriging with the training data $\{(\vy_k,\hat{G}_k)\}$.
\item \textbf{Evaluate the response surface:} For a point $\vx\in\sX$, compute
\begin{equation}
f(\vx)\;\approx\;\tilde{F}(\vx)\;=\;\tilde{G}(\mW_1^T\vx),
\end{equation}
where $\tilde{G}$ is the trained kriging surface.
\end{enumerate}
We conclude this section with summarizing remarks. First, the choice of $n$ requires some interaction from the user. We advocate such interaction since we have found that one can uncover insights into the function $f$ by examining the $\lambda_i$ and the elements of $\vw_i$. Second, as written, there is substantial freedom in choosing both the design sites on the reduced domain and the response surface. This was intentional. For our purposes, it suffices to use the gradient analysis to construct an approximation on the active subspace, but the details of the approximation will require many more practical considerations than we can address here. We have chosen kriging primarily because of (i) the natural fit of the computed $\lambda_i$ to the correlation length parameters and training data noise model, and (ii) its flexibility with scattered design sites. However, many other options for approximation are possible including global polynomials, regression splines, or finite element approximations; Russi advocates a global quadratic polynomial~\cite{Russi2010} on the subspace.

Third, with the gradient available for $f$, one could use \eqref{eq:chainrule} to obtain gradients with respect to the reduced coordinates. This could then be used to improve the response surface on the active subspace~\cite{Koehler1996}. As also mentioned in~\cite{Koehler1996}, since we know a great deal about our correlation function, we could create designs that satisfy optimality criteria such as maximum entropy. 

Finally, we note that the function evaluations $f_j$ could be better used to construct the response surface on the subspace. We have intentionally avoided proposing any strategies for such use; we prefer instead to use them as a testing set for the response surface as detailed in the next section. 

%% file: sec4-example.tex
\section{Numerical example}
\label{sec:example}
In this numerical exercise, we study an elliptic PDE with a random field model for the coefficients. Such problems are common test cases for methods in uncertainty
quantification~\cite{babuska2004galerkin}. MATLAB codes for this study can be found at \url{https://bitbucket.org/paulcon/active-subspace-methods-in-theory-and-practice}. They require MATLAB's PDE Toolbox, the random field simulation code at \url{http://www.mathworks.com/matlabcentral/fileexchange/27613-random-field-simulation}, and the Gaussian process regression codes at \url{http://www.gaussianprocess.org/gpml/code/matlab/doc/}.

\subsection{Forward and adjoint problem}
Consider the following linear elliptic PDE. Let $u=u(\vs,\vx)$ satisfy
\begin{equation}
-\nabla_\vs\cdot(a\,\nabla_\vs u) \;=\; 1, \qquad \vs\in[0,1]^2.
\end{equation}
We set homogeneous Dirichlet boundary conditions on the left, top, and bottom of the spatial domain; denote this boundary by $\Gamma_1$. The right side of the spatial domain denoted $\Gamma_2$ has a homogeneous Neumann boundary condition. The log of the coefficients $a=a(\vs,\vx)$ of the differential operator are given by a truncated Karhunen-Loeve (KL) type expansion
\begin{equation}
\label{eq:kl}
\log(a(\vs,\vx)) \;=\; \sum_{i=1}^m x_i\,\gamma_i\,\phi_i,
\end{equation}
where the $x_i$ are independent, identically distributed standard normal random variables, and the $\{\phi_i,\sigma_i\}$ are the eigenpairs of the correlation operator
\begin{equation}
\label{eq:corr}
\mathcal{C}(\vs,\vt) \;=\; \exp\left(\beta^{-1}\,\|\vs-\vt\|_1\right).
\end{equation}
We will study the quality of the active subspace approximation for two correlation lengths, $\beta=1$ and $\beta=0.01$. We choose a truncation of the field $m=100$, which implies that the parameter space $\sX=\mathbb{R}^{100}$ with $\rho$ a standard Gaussian density function. Define the linear function of the solution 
\begin{equation}
\frac{1}{|\Gamma_2|} \int_{\Gamma_2} u(\vs,\vx) \,d\vs.
\end{equation}
This is the function we will study with the active subspace method.

Given a value for the input parameters $\vx$, we discretize the elliptic problem with a standard linear finite element method using MATLAB's PDE Toolbox. The discretized domain has 34320 triangles and 17361 nodes; the eigenfunctions $\phi_i=\phi_i(\vx)$ from \eqref{eq:kl} are approximated on this mesh. The matrix equation for the
discrete solution $\vu=\vu(\vx)$ at the mesh nodes is
\begin{equation}
\label{eq:fem}
\mK\vu = \vf, 
\end{equation}
where $\mK=\mK(\vx)$ is symmetric and positive definite for all $\vx\in\sX$. We can
approximate the linear functional as
\begin{equation}
\label{eq:linearfunctional}
f(\vx) \;=\; \vc^T\mM\vu(\vx) \;\approx\;
\frac{1}{|\Gamma_2|} \int_{\Gamma_2} u(\vs,\vx) \,d\vs.
\end{equation}
where $\mM$ is the symmetric mass matrix, and the components of $\vc$ corresponding to mesh nodes on $\Gamma_2$ are equal to one with the rest equal to zero.

Since the quantity of interest can be written as a linear functional of the solution, we can define adjoint variables that enable us to compute $\nabla_\vx f$,
\begin{equation}
\label{eq:adjf}
f \;=\; \vc^T\mM\vu \;=\; \vc^T\mM\vu - \vy^T(\mK\vu - \vf),
\end{equation}
for any constant vector $\vy$. Taking the derivative of \eqref{eq:adjf} with respect to the input $x_i$, we get
\begin{equation}
\frac{\partial f}{\partial x_i} 
 \;=\; 
\vc^T\mM\left(\frac{\partial \vu}{\partial x_i}\right)
- \vy^T\left(\frac{\partial\mK}{\partial x_i}\vu + \mK\frac{\partial\vu}{\partial x_i}\right)
\;=\; 
\left(\vc^T\mM - \vy^T\mK\right)\left(\frac{\partial \vu}{\partial x_i}\right)
- \vy^T\left(\frac{\partial\mK}{\partial x_i}\right)\vu  
\end{equation}
If we choose $\vy$ to solve the adjoint equation $\mK^T\vy=\mM^T\vc$, then
\begin{equation}
\label{eq:derf}
\frac{\partial f}{\partial x_i} = - \vy^T\left(\frac{\partial\mK}{\partial x_i}\right)\vu.
\end{equation} 
To approximate the gradient $\nabla_\vx f$ at the point $\vx$, we compute the finite element solution with \eqref{eq:fem}, solve the adjoint problem, and compute the components with \eqref{eq:derf}. The derivative of $\mK$ with respect to $x_i$ is straightforward to compute from the derivative of $a(\vs,\vx)$ and the same finite element discretization. 

\subsection{Dimension reduction study}
Next we apply the dimension reduction method to the quantity of interest $f$ from \eqref{eq:linearfunctional}. We compare two cases: (i) the random field model for $a$  has a long correlation length ($\beta=1$ in \eqref{eq:corr}), which corresponds to rapidly decaying KL singular values $\gamma_i$ in \eqref{eq:kl}, and (ii) $a$ has a short correlation length ($\beta=0.01$), which corresponds to slowly decaying KL singular values. The number of terms in the KL series is often chosen according to the decay of the singular values, e.g., to capture some proportion of the energy in $a$; we choose $m=100$ in both cases for illustration purposes. Thus, the linear function $f$ of the solution to the PDE $u$ is parameterized by the $m=100$ parameters characterizing the elliptic coefficients $a$. 

We use the finite element model and its adjoint to compute $f$ and $\nabla_\vx f$ at $M=300$ points drawn from a $m$-variate normal distribution with zero mean and identity covariance. Table \ref{tab:svals} lists the first five singular values from the SVD of the matrix of gradient samples (labeled ASM for \emph{active subspace method}) for both $\beta=1$ and $\beta=0.01$. They are normalized so that the first singular value is 1. We compare these with normalized versions of the KL singular values for both correlation lengths. Notice that the slow decay of the KL singular values for $\beta=0.01$ suggests that little dimension reduction is possible. However, the singular values of $G$ decay very rapidly, which suggests dimension reduction will be effective---assuming that the Monte Carlo approximation of the matrix $\mC$ is sufficiently accurate. The comparison is only meaningful in the sense of dimension reduction; the singular values of $G$ are with respect to a specific scalar quantity of interest while the KL singular values are for the spatially varying field. Nevertheless, the conclusions drawn from these decay rates are comparable. 

\begin{table}[ht]
\caption{This table compares the normalized KL singular values from \eqref{eq:kl} with the normalized singular values from the active subspace analysis, i.e., the singular values of $G$ from \eqref{eq:svd}. Two values of $\beta$ are the correlation length parameters for the random field model of the elliptic coefficients from \eqref{eq:corr}.}
\centering
\begin{tabular}{c|c|c|c}
\toprule
KL, $\beta=1$ & ASM, $\beta=1$ & KL, $\beta=0.01$ & ASM, $\beta=0.01$ \\
\midrule
1.0000 & 1.0000 & 1.0000 & 1.0000\\
0.0978 & 0.0010 & 0.9946 & 0.0055\\
0.0975 & 0.0006 & 0.9873 & 0.0047\\
0.0282 & 0.0005 & 0.9836 & 0.0046\\
0.0280 & 0.0002 & 0.9774 & 0.0042\\
\bottomrule
\end{tabular}
\label{tab:svals}
\end{table}

Figure \ref{fig:evecs} plots the components of the first two eigenvectors from the active subspace analysis---i.e., the first two columns of $\mW$ from \eqref{eq:svd}---for both correlation lengths. Notice that the mass is more evenly distributed for the short correlation length. This is not surprising. The magnitude of the components of the first eigenvector are a measure of sensitivity of the function $f$ to perturbations in the parameters. The relative clustering of large values toward smaller indices (the left side of the plot) for the longer correlation length implies that the coefficients $x_i$ in the KL series with larger singular values contribute the most to the variability in $f$. But this relationship relaxes for shorter correlation lengths. 

\begin{figure}[ht]
\centering
\subfloat[]{
\includegraphics[width=0.48\linewidth]{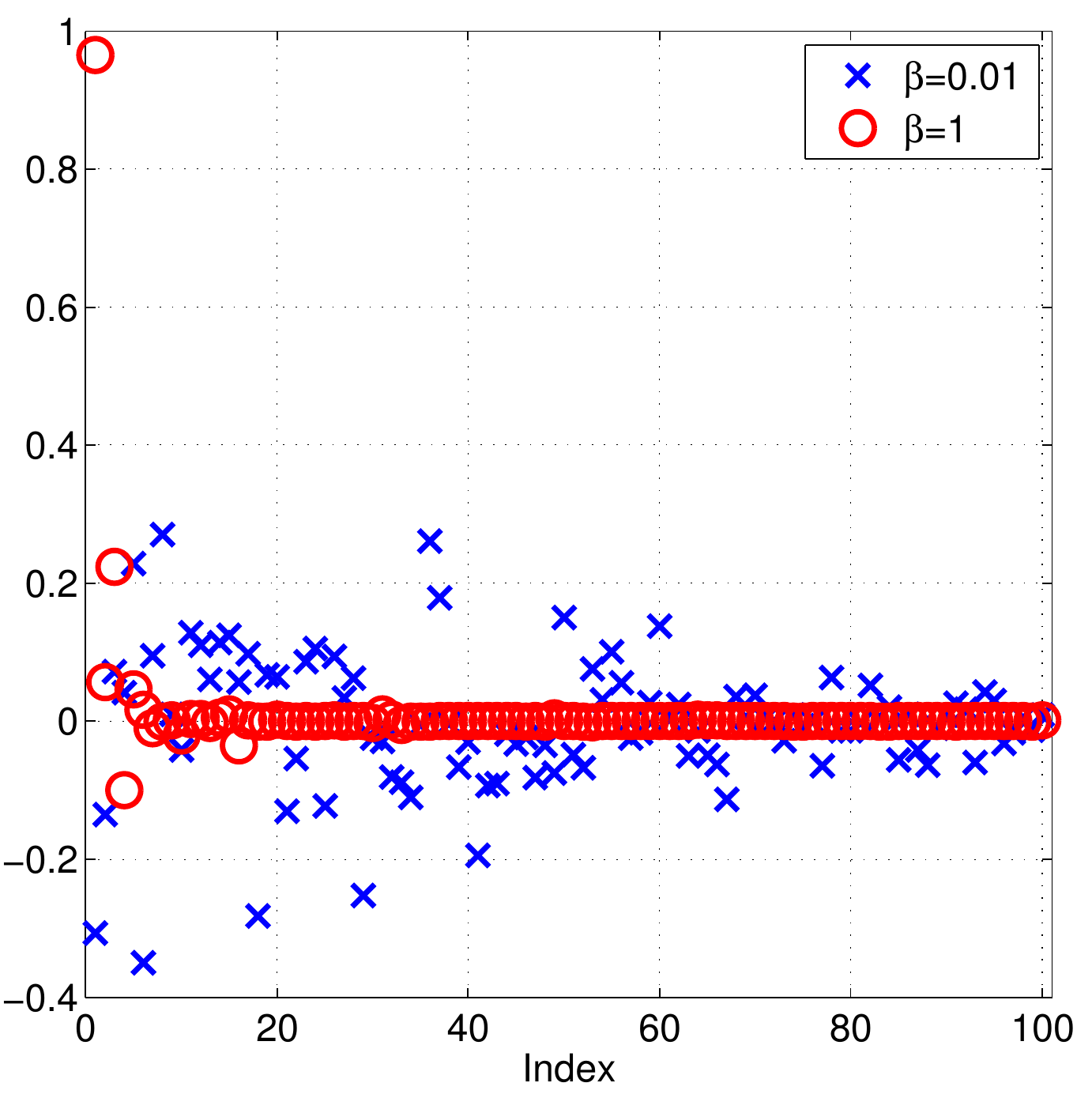}
}
\subfloat[]{
\includegraphics[width=0.48\linewidth]{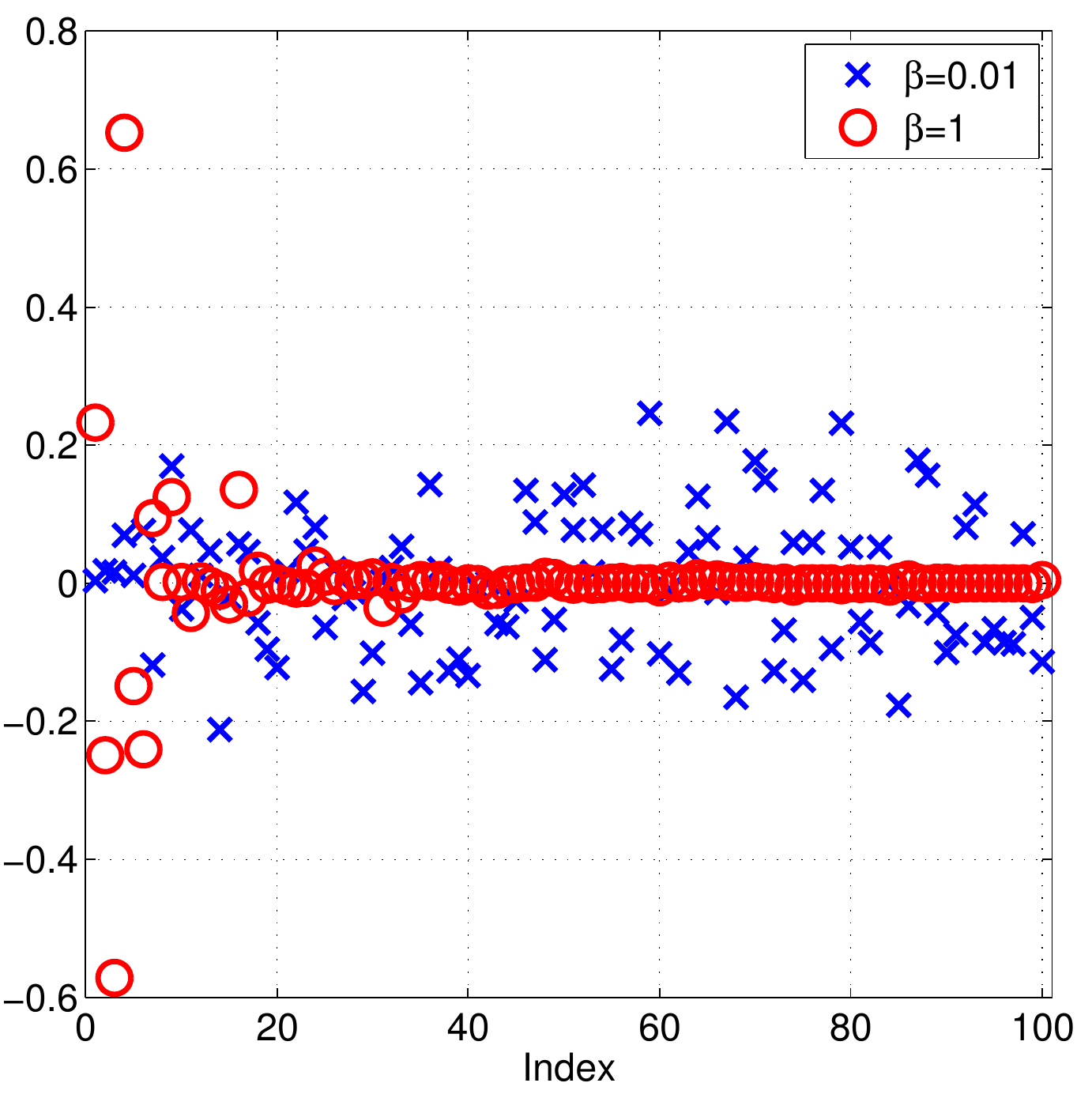}
}
\caption{The left figure shows the components of the first eigenvector from the active subspace analysis for both long ($\beta=1$) and short ($\beta=0.01$) correlation lengths. The right figure shows the components of the second eigenvector. (Colors are visible in the electronic version.) }
\label{fig:evecs}
\end{figure}

\begin{table}[ht]
\caption{This table shows the average relative error of the kriging surface at 300 testing sites as the dimension of the subspace increases. We report this approximation error for $\beta=1$ and $\beta=0.01$ in \eqref{eq:corr}. Note the relatively slow decay of the error justifies our attention on active subspaces for $n=1$ and $n=2$.}
\centering
\begin{tabular}{c|c|c}
\toprule
$n$ & $\beta=0.01$ & $\beta=1$ \\
\midrule
1 & 7.88e-3 & 1.78e-1 \\
2 & 7.82e-3 & 1.49e-1 \\
3 & 7.57e-3 & 1.88e-1 \\
4 & 6.75e-3 & 1.22e-1 \\
5 & 6.61e-3 & 1.10e-1 \\
\bottomrule
\end{tabular}
\label{tab:errs}
\end{table}

Table \ref{tab:errs} studies the approximation quality as the dimension of the low-dimensional space increases for both choices of the correlation length parameter $\beta$. The numbers represent the average relative error in the active subspace method's approximation for the $M=300$ testing evaluations. We trained the kriging surfaces using a five-point tensor product design on the reduced domain; the univariate designs use the points $\{-3,-1.5,0,1.5,3\}$ to cover three standard deviations in the input space. We use the heuristic described in Section \ref{sec:training} for choosing the kriging hyperparameters. The average error does not decrease rapidly for increasing $n$. Therefore, we will only consider the active subspace approximations for $n=1$ and $n=2$. 

We compare the accuracy of the kriging surface on the active subspace with a kriging surface on the one- and two-dimensional coordinate subspaces defined by the largest-in-magnitude components of the gradient $\nabla_\vx f$ evaluated at the origin. In other words, we compare the approximation on the active subspace to an approximation using local sensitivity analysis to reduce the number of parameters. For the long correlation length $\beta=1$, the two most important coordinates are $x_1$ and $x_3$. For the short correlation length $\beta=0.01$, the two most important coordinates are $x_6$ and $x_1$. For the coordinate dimension reduction, we use maximum likelihood method implemented in the GPML code~\cite{Rasmussen2006} with a maximum of 500 function evaluations to tune the hyperparameters of an isotropic squared-exponential covariance kernel and a quadratic polynomial basis. 

\begin{figure}[ht]
\centering
\subfloat[ASM, $\beta=1$]{
\includegraphics[width=0.48\linewidth]{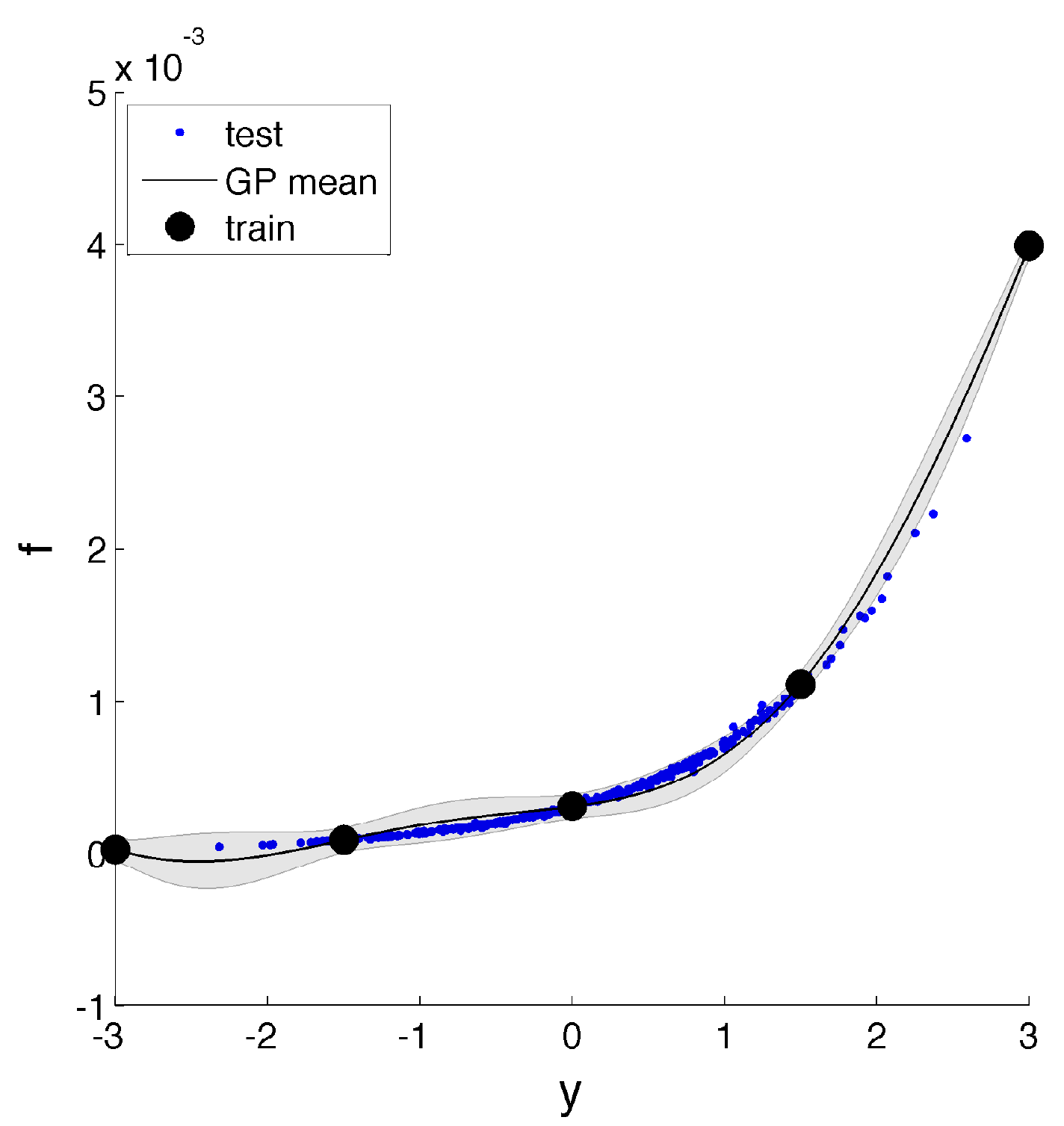}
\label{fig:1d1}
}
\subfloat[ASM, $\beta=0.01$]{
\includegraphics[width=0.48\linewidth]{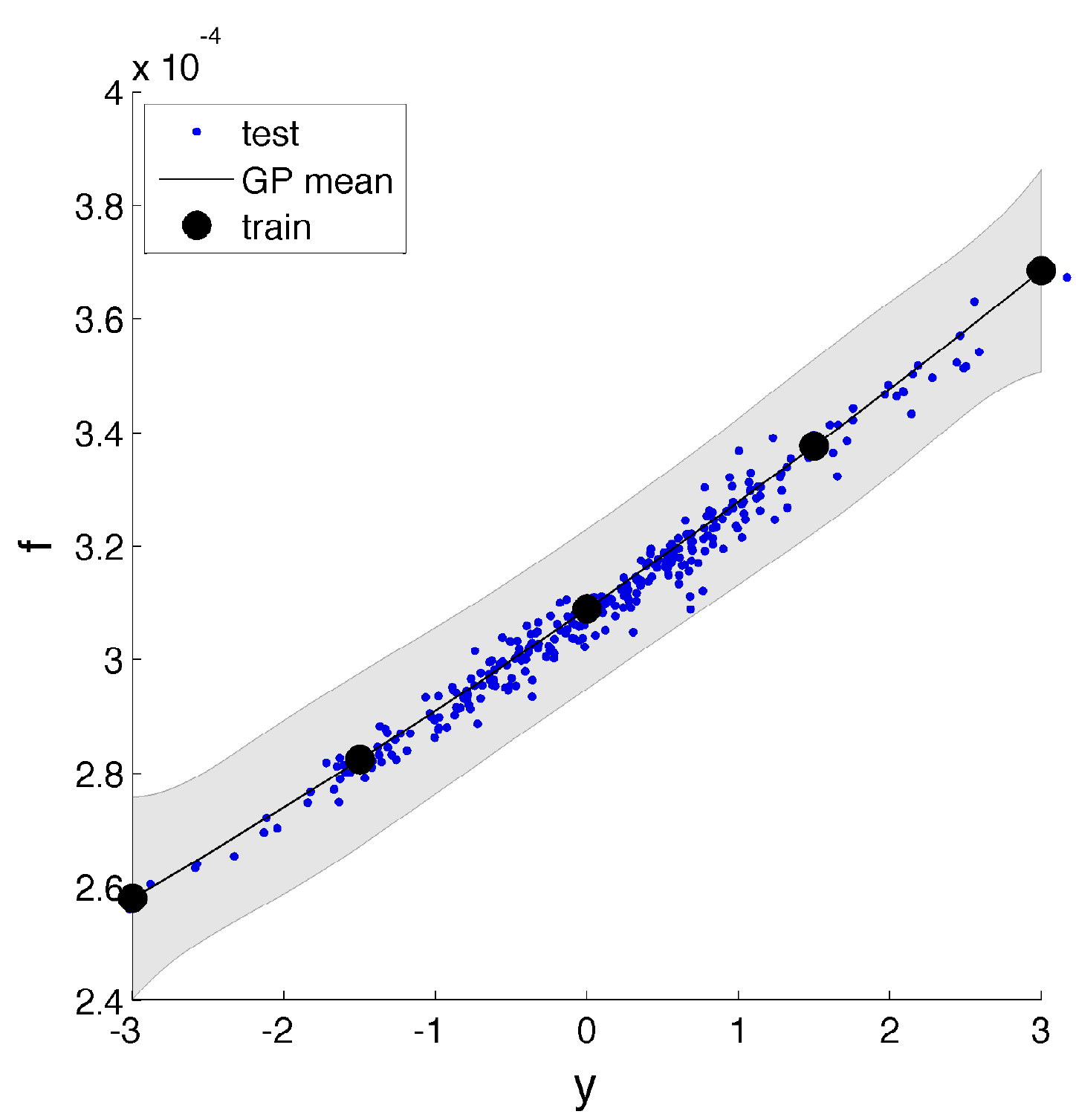}
\label{fig:1d2}
}\\
\subfloat[SENS, $\beta=1$]{
\includegraphics[width=0.48\linewidth]{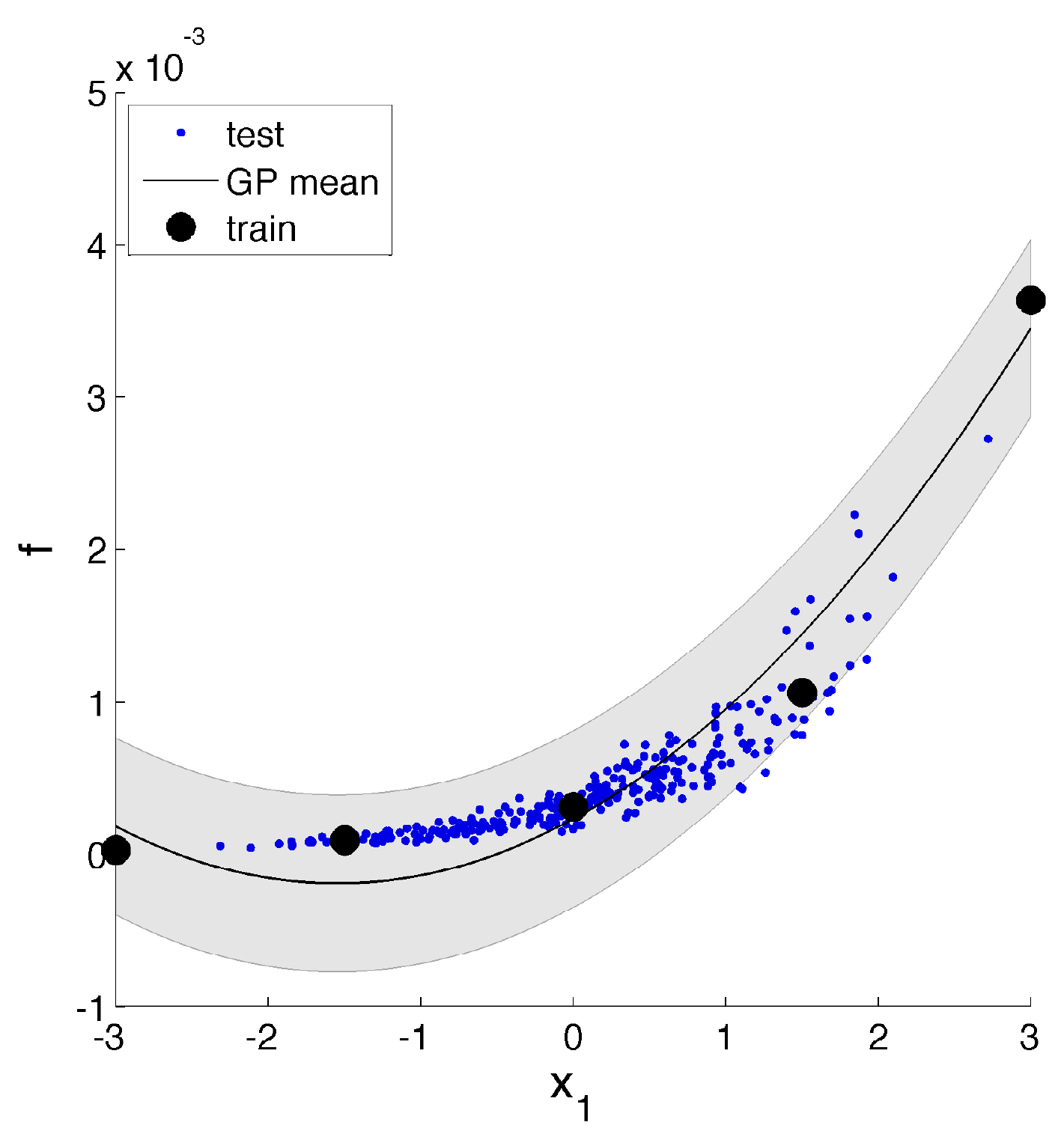}
\label{fig:1d3}
}
\subfloat[SENS, $\beta=0.01$]{
\includegraphics[width=0.48\linewidth]{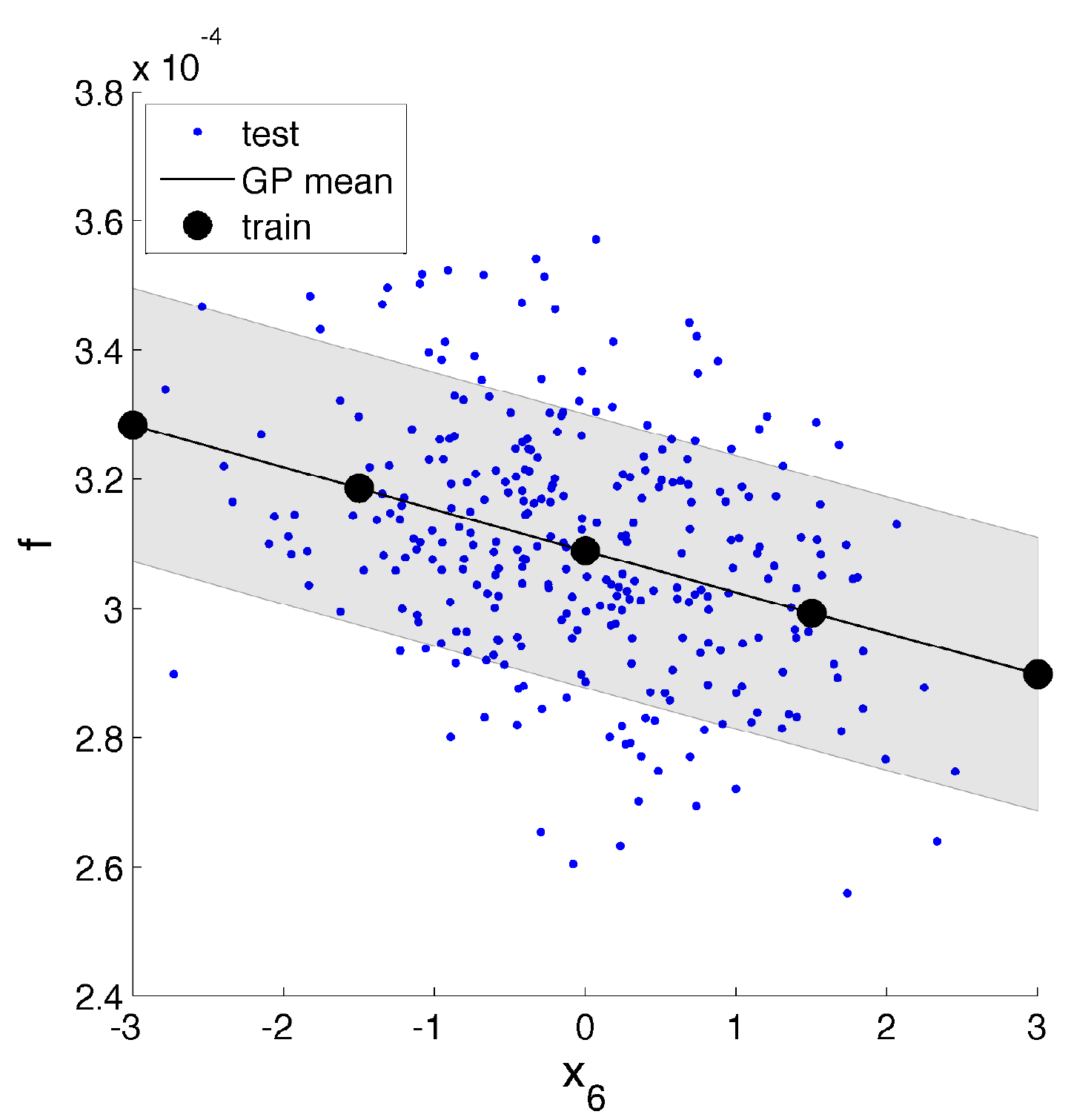}
\label{fig:1d4}
}
\caption{Comparing the kriging surfaces constructed on the one-dimensional subspaces defined by the active subspaces (\ref{fig:1d1} and \ref{fig:1d2}) and a local sensitivity analysis at the origin (\ref{fig:1d3} and \ref{fig:1d4}). (Colors are visible in the electronic version.) }
\label{fig:1d}
\end{figure}

Figure \ref{fig:1d} shows the one-dimensional projections. In all subfigures, the solid black line is the mean kriging prediction, and the gray shaded region is the two standard deviation confidence interval. The blue dots show the 300 evaluations of $f$ computed while studying the gradients projected onto the subspaces; we use these evaluations as testing data. Subfigures \ref{fig:1d1} and \ref{fig:1d3} show the approximation for the long correlation length. Notice how the function evaluations cluster more tightly around the mean prediction for the active subspace. Loosely speaking, this means the active subspace has found the right angle from which to view the high-dimensional data to uncover its one-dimensional character. Subfigures \ref{fig:1d2} and \ref{fig:1d4} show the same plots for the shorter correlation length. Notice how for both methods the spread of the function evaluations is larger. However, the active subspace is again able to uncover a strongly dominant direction; when viewed from this direction, the data is essentially one-dimensional. The local sensitivity method reveals no such trend for the shorter correlation length. Figure \ref{fig:2d} shows the same plots for the two-dimensional subspaces without the shaded regions for the confidence intervals. The conclusions drawn from the one-dimensional plots are the same for the two-dimensional plots.

\begin{figure}[ht]
\centering
\subfloat[ASM, $\beta=1$]{
\includegraphics[width=0.48\linewidth]{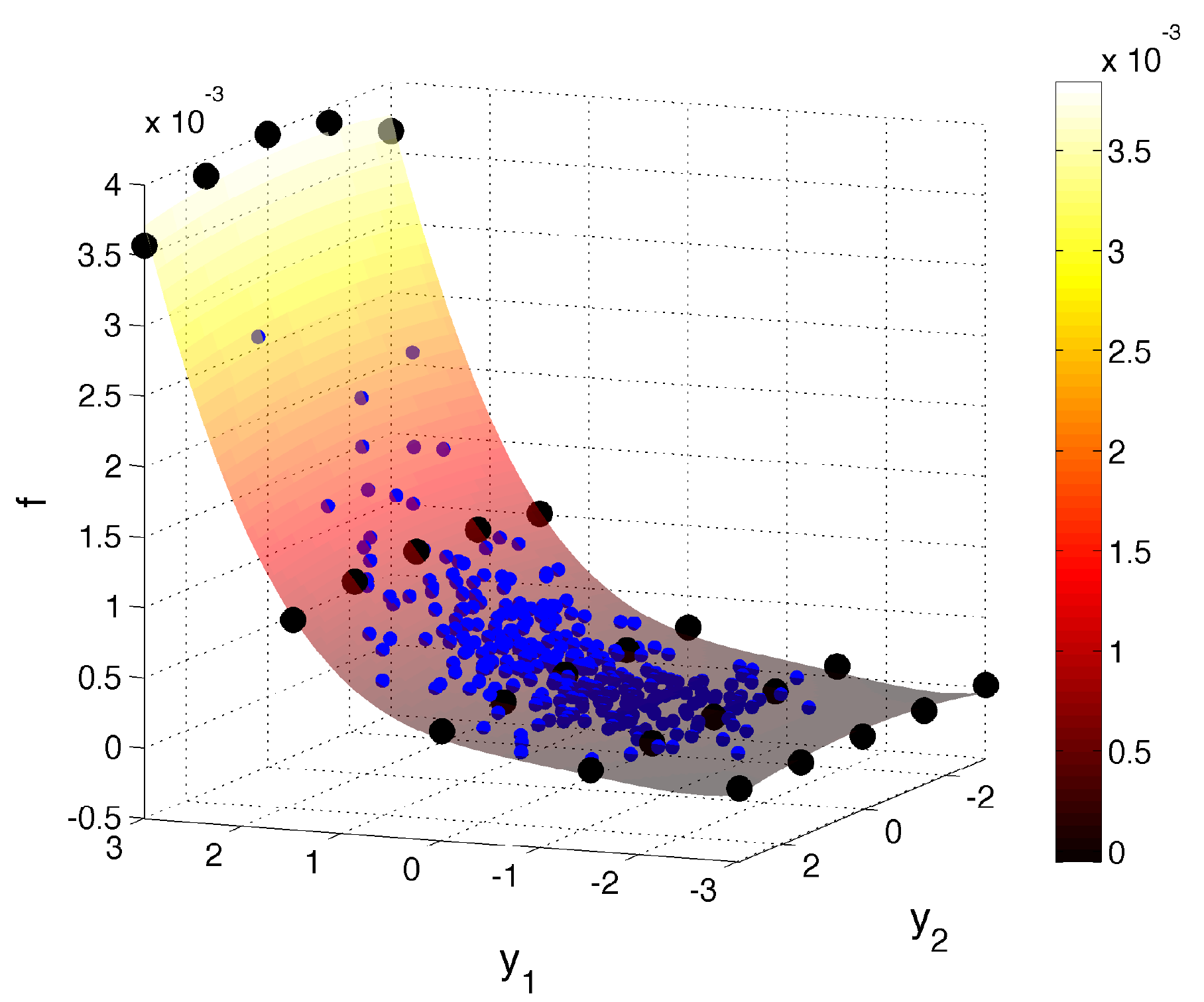}
\label{fig:2d1}
}
\subfloat[ASM, $\beta=0.01$]{
\includegraphics[width=0.48\linewidth]{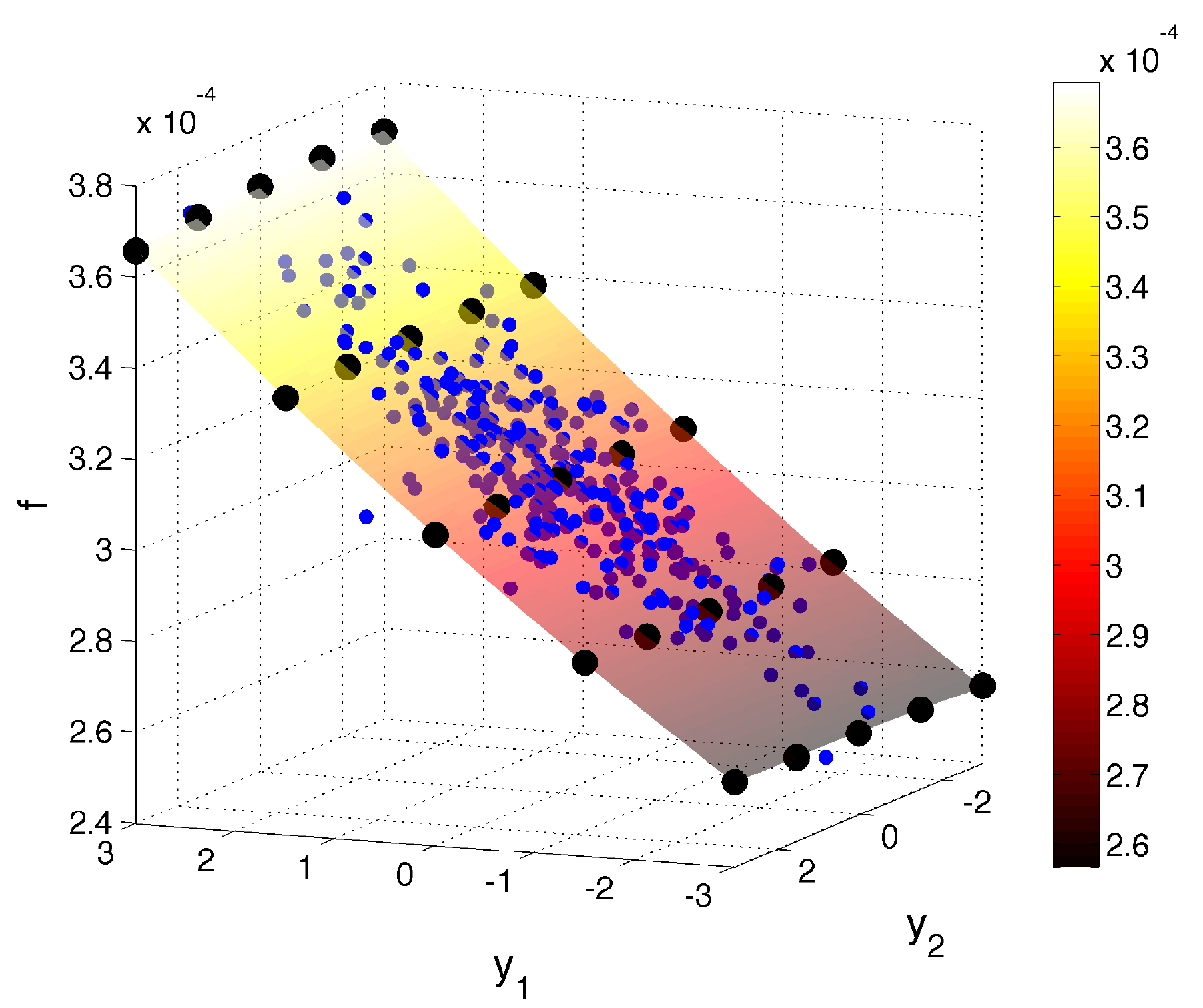}
\label{fig:2d2}
}\\
\subfloat[SENS, $\beta=1$]{
\includegraphics[width=0.48\linewidth]{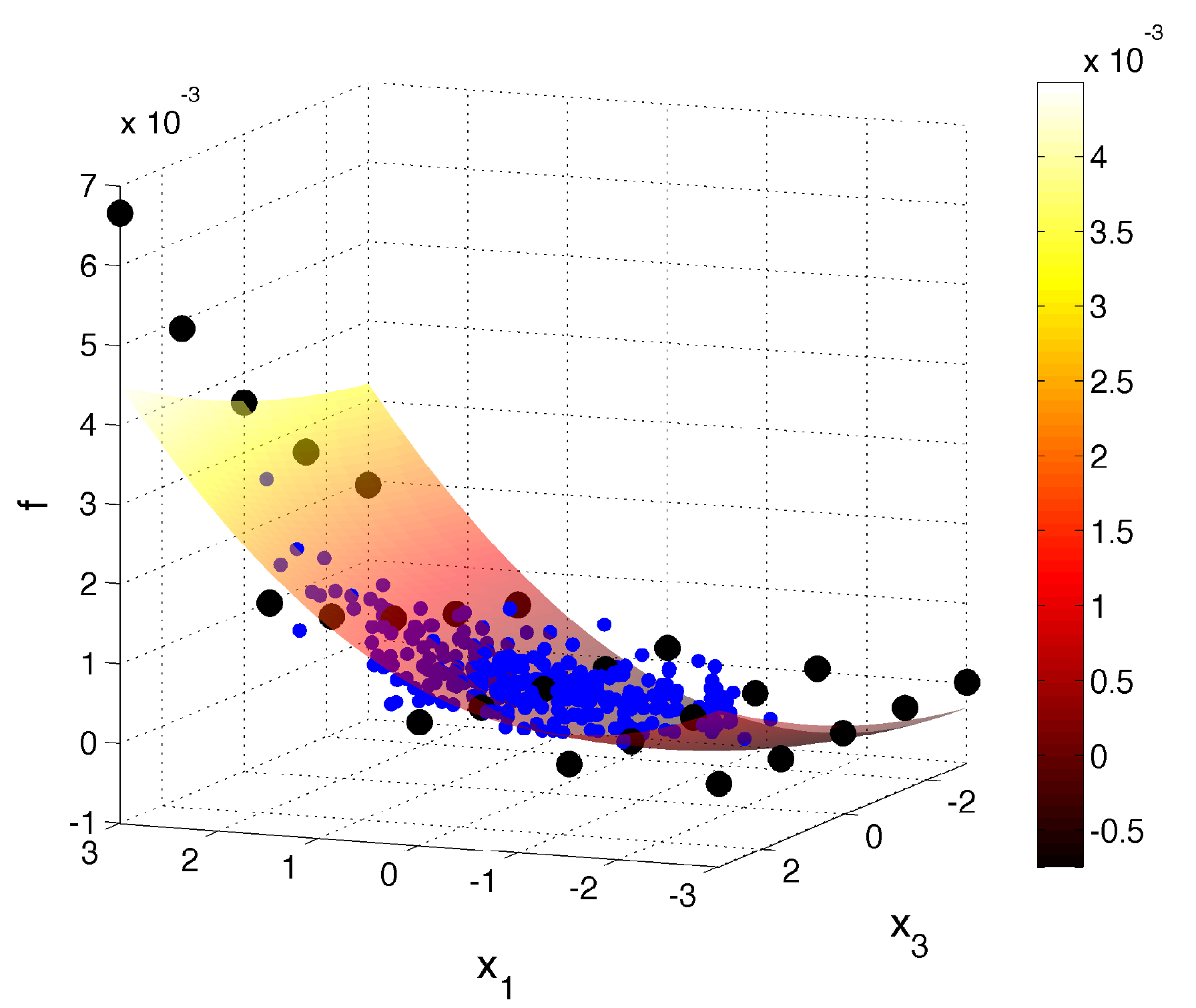}
\label{fig:2d3}
}
\subfloat[SENS, $\beta=0.01$]{
\includegraphics[width=0.48\linewidth]{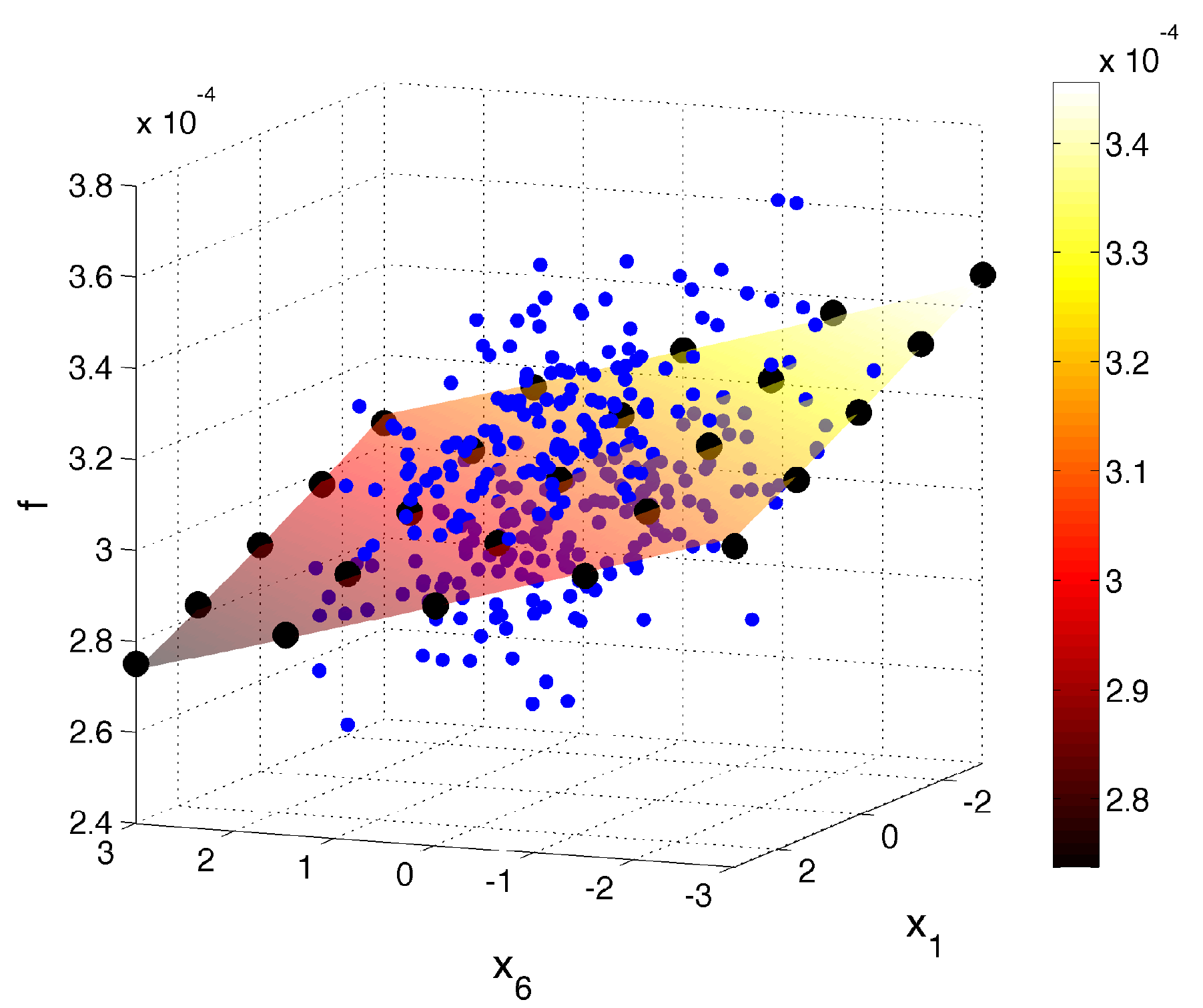}
\label{fig:2d4}
}
\caption{Kriging surfaces constructed on the two-dimensional subspaces defined by the active subspaces (\ref{fig:2d1} and \ref{fig:2d2}) and a local sensitivity analysis at the origin (\ref{fig:2d3} and \ref{fig:2d4}). (Colors are visible in the electronic version.) }
\label{fig:2d}
\end{figure}

Figure \ref{fig:err} shows histograms of the log of the relative error in the testing data for the two correlation lengths on both the one- and two-dimensional subspace approximations. For each case, the histogram of the testing error in the active subspace approach is compared with the coordinate reduction approach. In all cases, the active subspace approach performs better as indicated by the leftward shift in the histogram, which corresponds to smaller error. 

\begin{figure}[ht]
\centering
\subfloat[$n=1$, $\beta=1$ ]{
\includegraphics[width=0.48\linewidth]{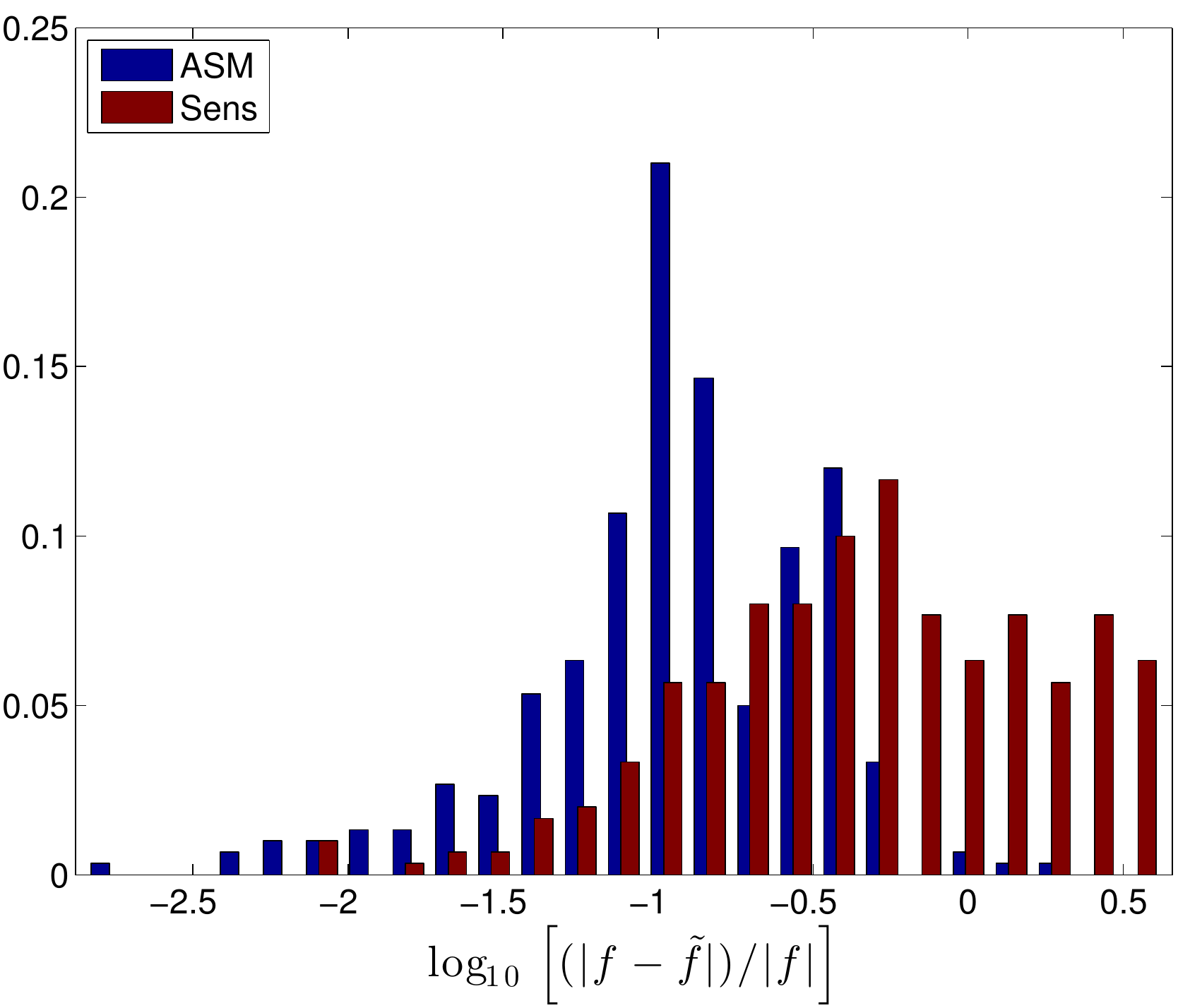}
\label{fig:err1}
}
\subfloat[$n=1$, $\beta=0.01$]{
\includegraphics[width=0.48\linewidth]{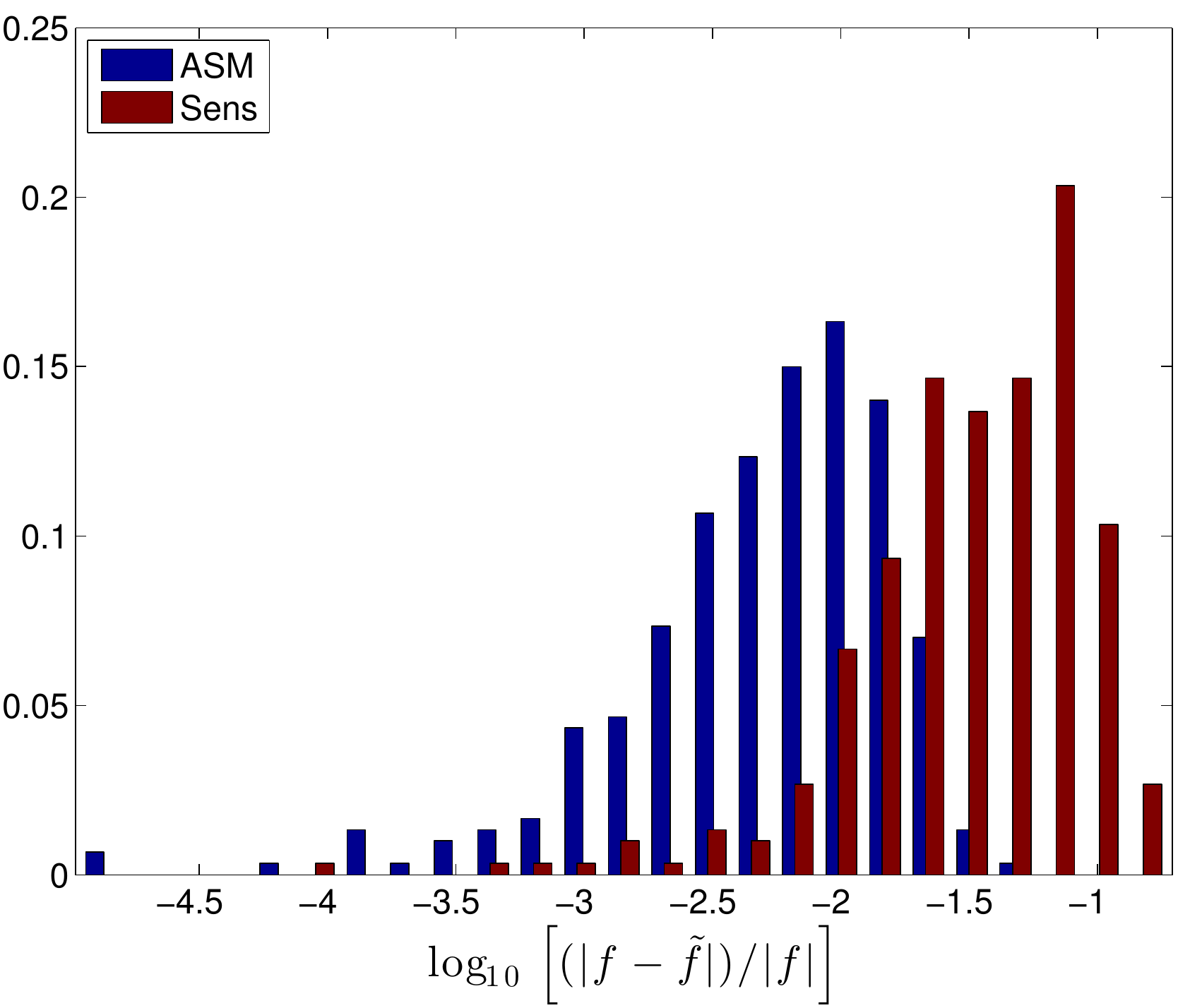}
\label{fig:err2}
}\\
\subfloat[$n=2$, $\beta=1$]{
\includegraphics[width=0.48\linewidth]{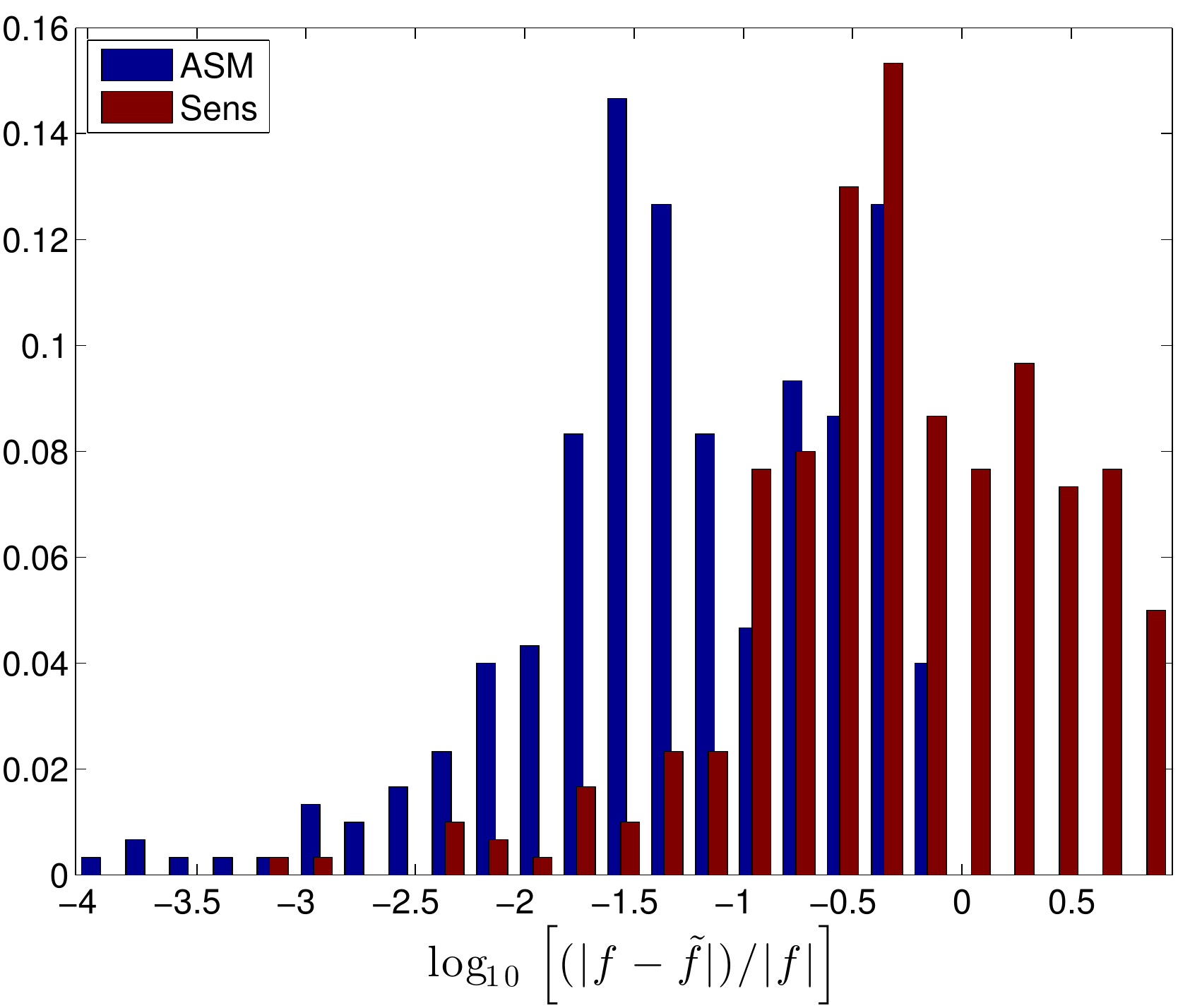}
\label{fig:err3}
}
\subfloat[$n=2$, $\beta=0.01$]{
\includegraphics[width=0.48\linewidth]{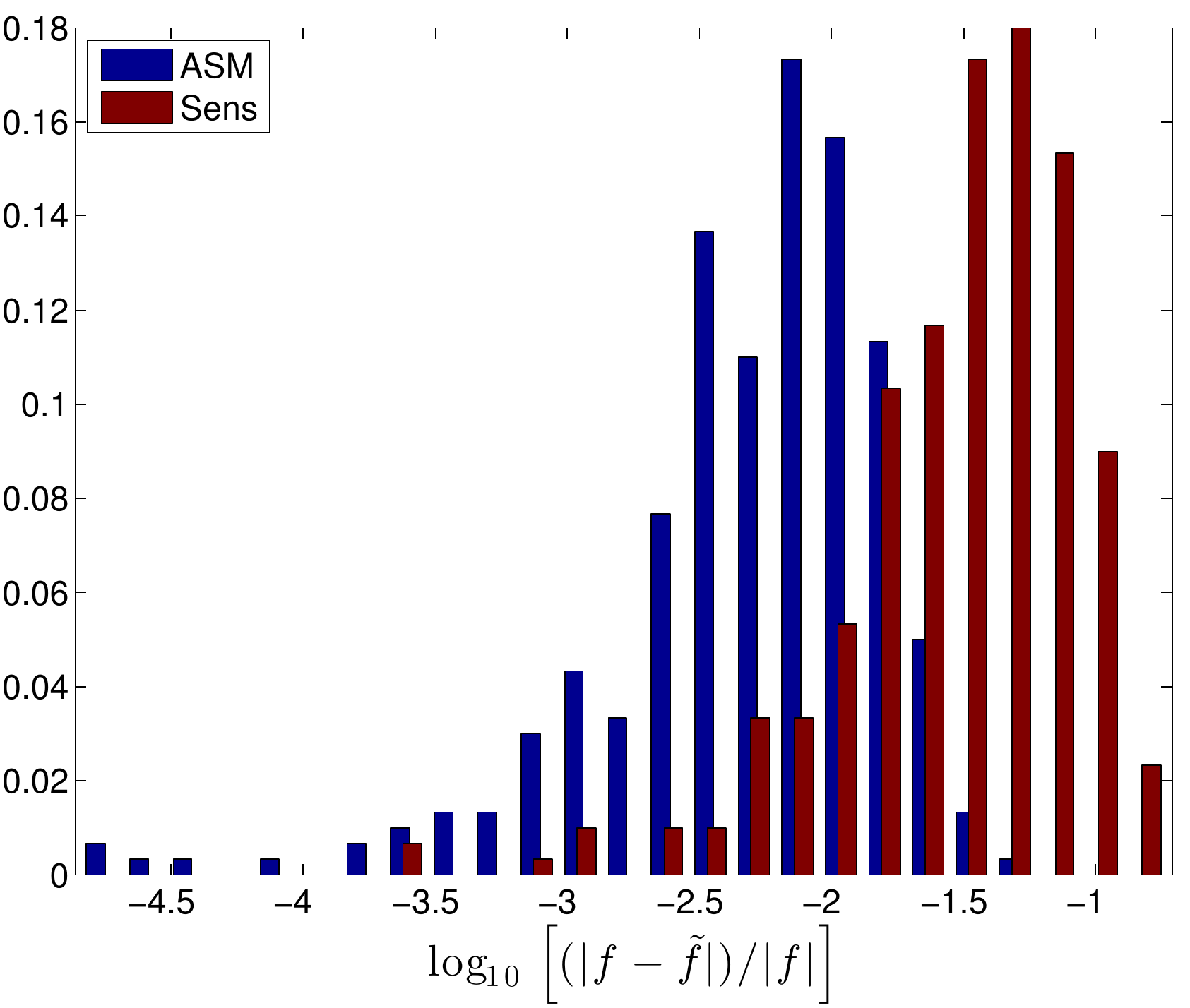}
\label{fig:err4}
}
\caption{Histograms of the log of the relative error in the testing data for kriging surfaces constructed on the one- and two-dimensional subspaces for the two correlation lengths. (Colors are visible in the electronic version.) }
\label{fig:err}
\end{figure}

We mention the costs of these two approaches in terms of the number function and gradient evaluations. The active subspace method used $M=300$ samples of the gradient to approximate the covariance matrix. Since we had no prior knowledge that $\mC$ would be low-rank, we chose $M=300$ as a $3\times$ oversampling rate given the $m=100$ variables. However, if we had suspected rapid decay in the singular values---and given that we would use at most only a two-dimensional subspace---we could have used many fewer gradient samples. With each gradient evaluation, we also get a function evaluation that we can use for testing the approximation. Given the eigenvectors defining the active subspace, we evaluated the function $P=5$ or $25$ more times for the one- and two-dimensional subspaces, respectively. We then tested the kriging surface on the subspace using the testing set computed along with the gradients. 

The local sensitivity method used one gradient evaluation to find the first and second most important input variable. It also used five or twenty-five additional function evaluations to train a kriging surface on one- and two-dimensional coordinate subspaces. It then used the same 300 function evaluations as testing data. Thus, the local method was significantly cheaper, but substantially less accurate. 

\subsection{Comparison with kriging on $\sX$}
Lastly, we compare the kriging surface constructed on the low-dimensional domain $\sY$ using the active subspace with a kriging surface on the full domain $\sX$. The cost of computing the gradient $\nabla_\vx f$ via adjoint computations is roughly twice the cost of computing the function $f$ for a particular $\vx$. Thus, the cost of constructing the active subspace approximation is roughly $3M+P$ function evaluations, where $M$ is the number of gradient samples, and $P$ is the number of evaluations for the design on $\sY$. 

For a fair comparison, we build a kriging surface on the $m$-dimensional space $\sX$ using $3M+P$ function evaluations. In this case $M=300$, $P=5$ for the one-dimensional subspace, and $P=25$ for the two-dimensional subspace. We evaluate $f$ at 500 additional points to create an independent testing set. Histograms of the testing errors are shown in Figure \ref{fig:fullerr}. For the same cost, the relative focus of the active subspace method produces a more accurate approximation than the response surface in $m=100$ dimensions.

\begin{figure}[ht]
\centering
\subfloat[$n=1$, $\beta=1$ ]{
\includegraphics[width=0.48\linewidth]{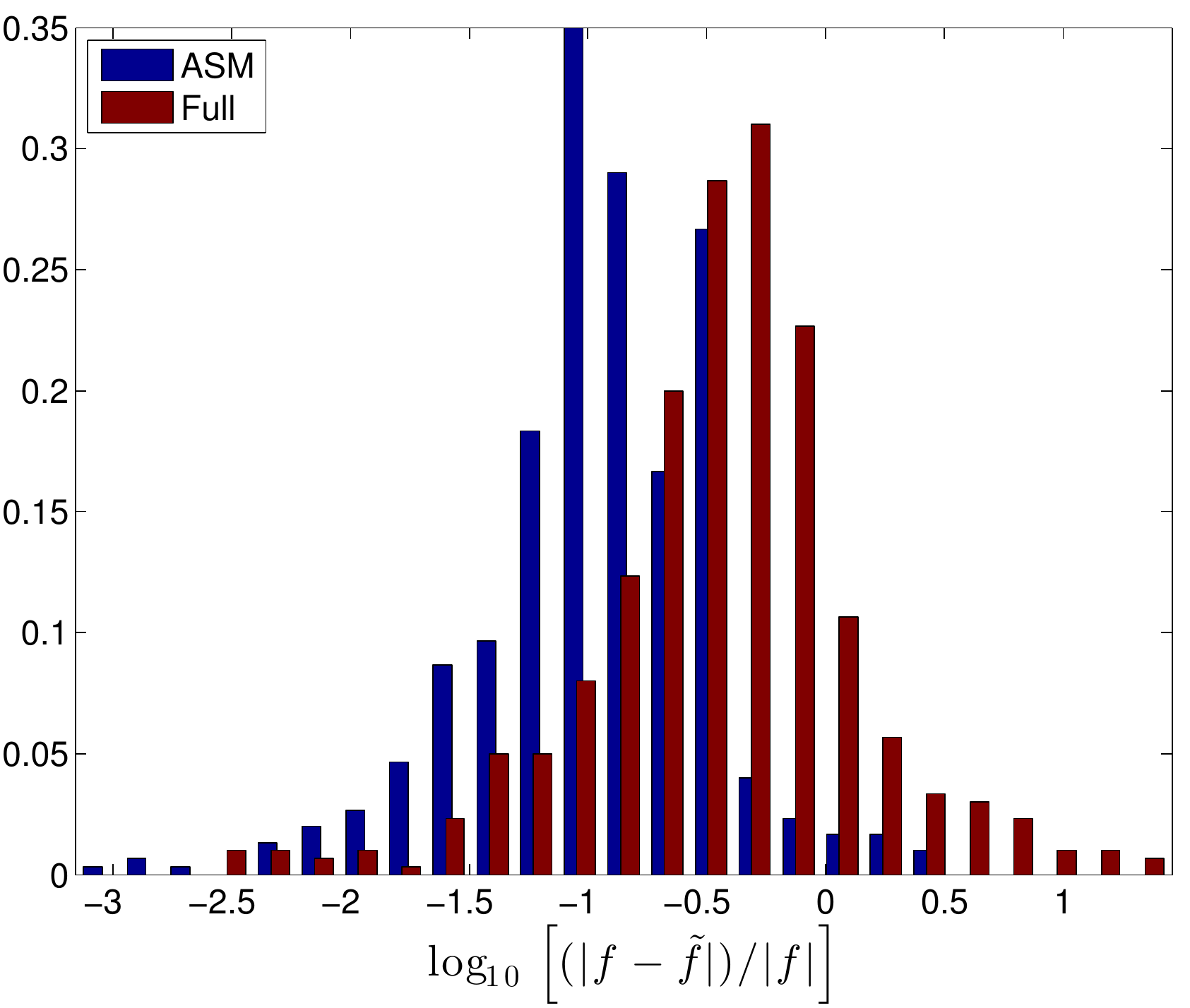}
\label{fig:fullerr1}
}
\subfloat[$n=1$, $\beta=0.01$]{
\includegraphics[width=0.48\linewidth]{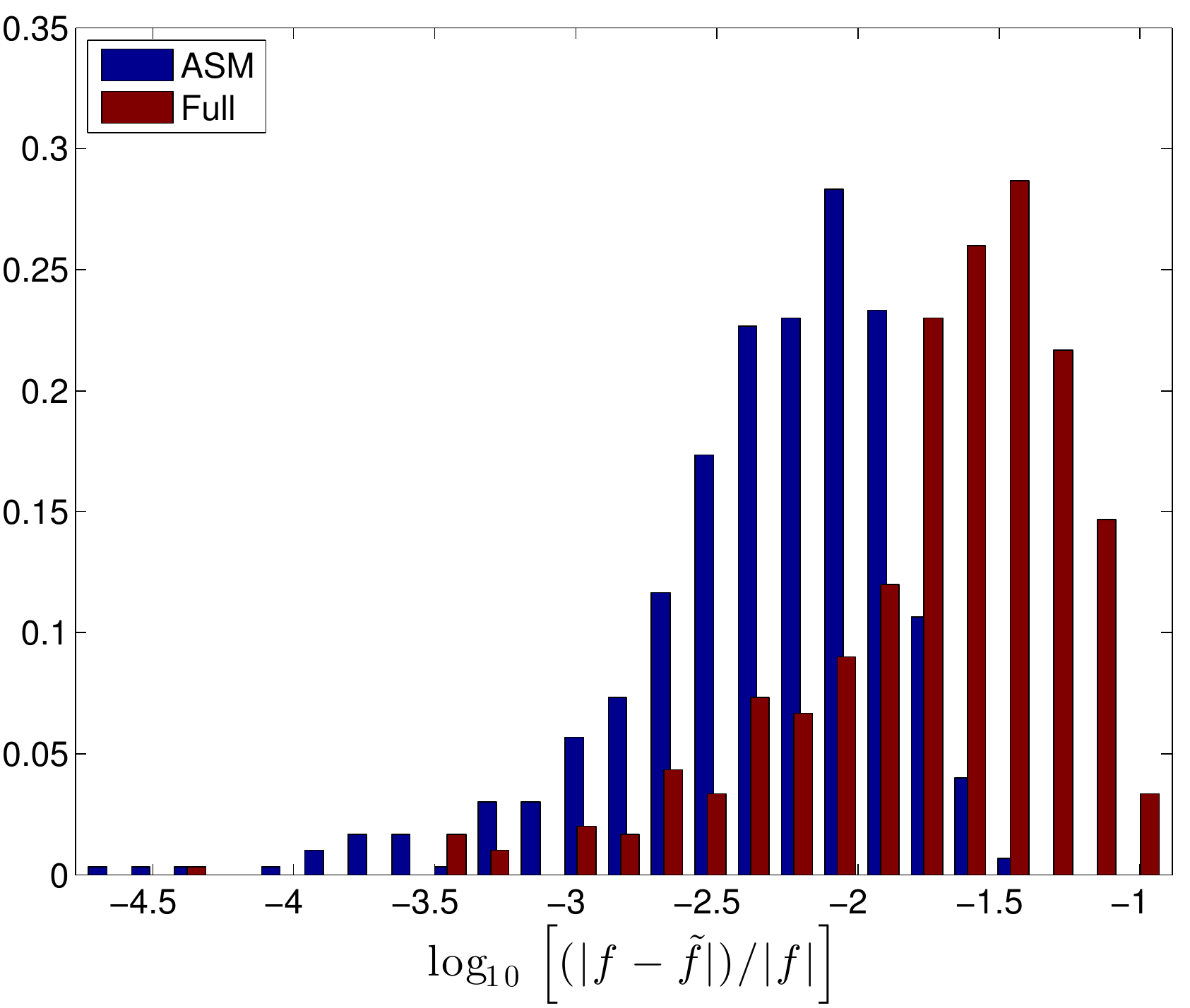}
\label{fig:fullerr2}
}\\
\subfloat[$n=2$, $\beta=1$]{
\includegraphics[width=0.48\linewidth]{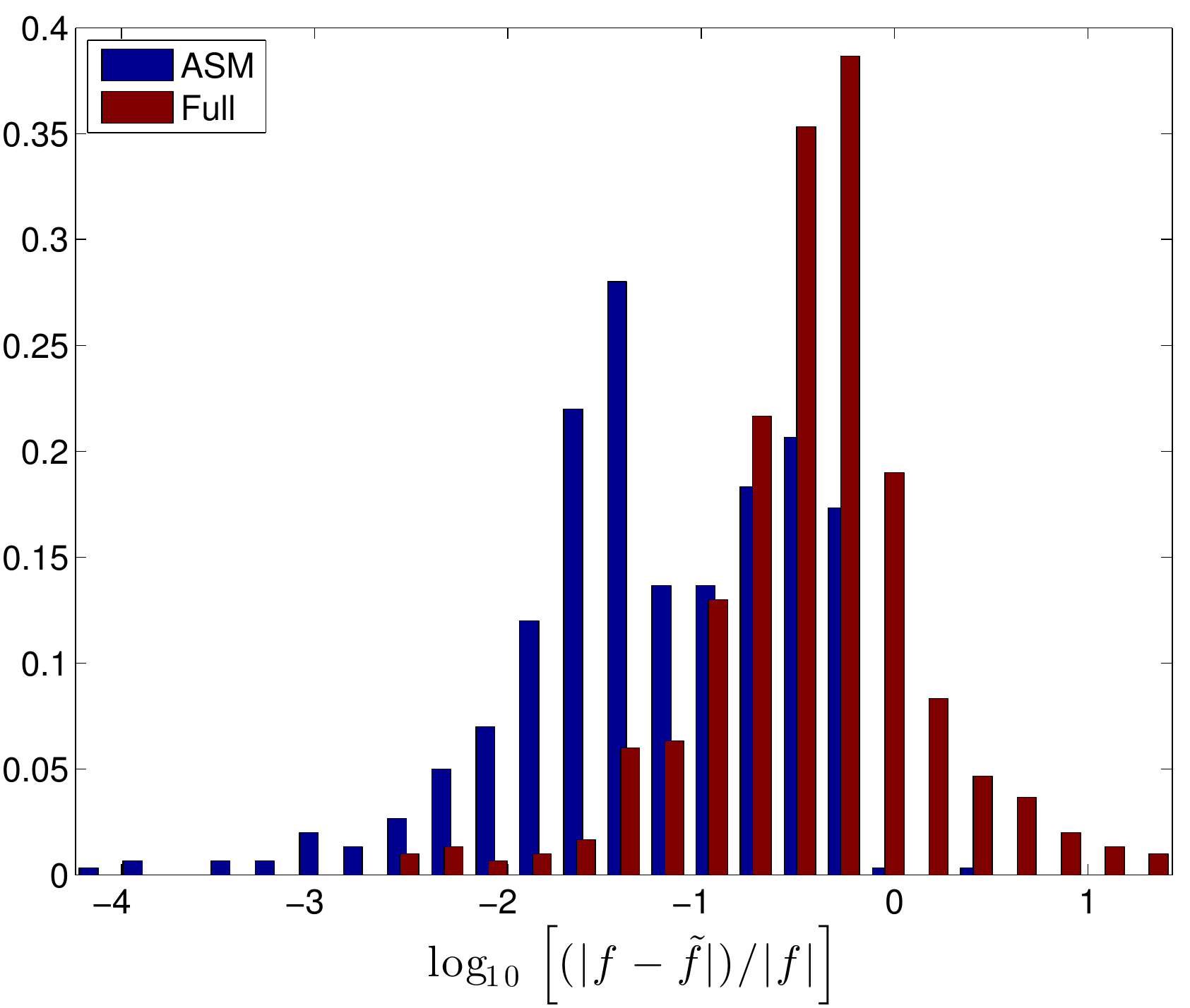}
\label{fig:fullerr3}
}
\subfloat[$n=2$, $\beta=0.01$]{
\includegraphics[width=0.48\linewidth]{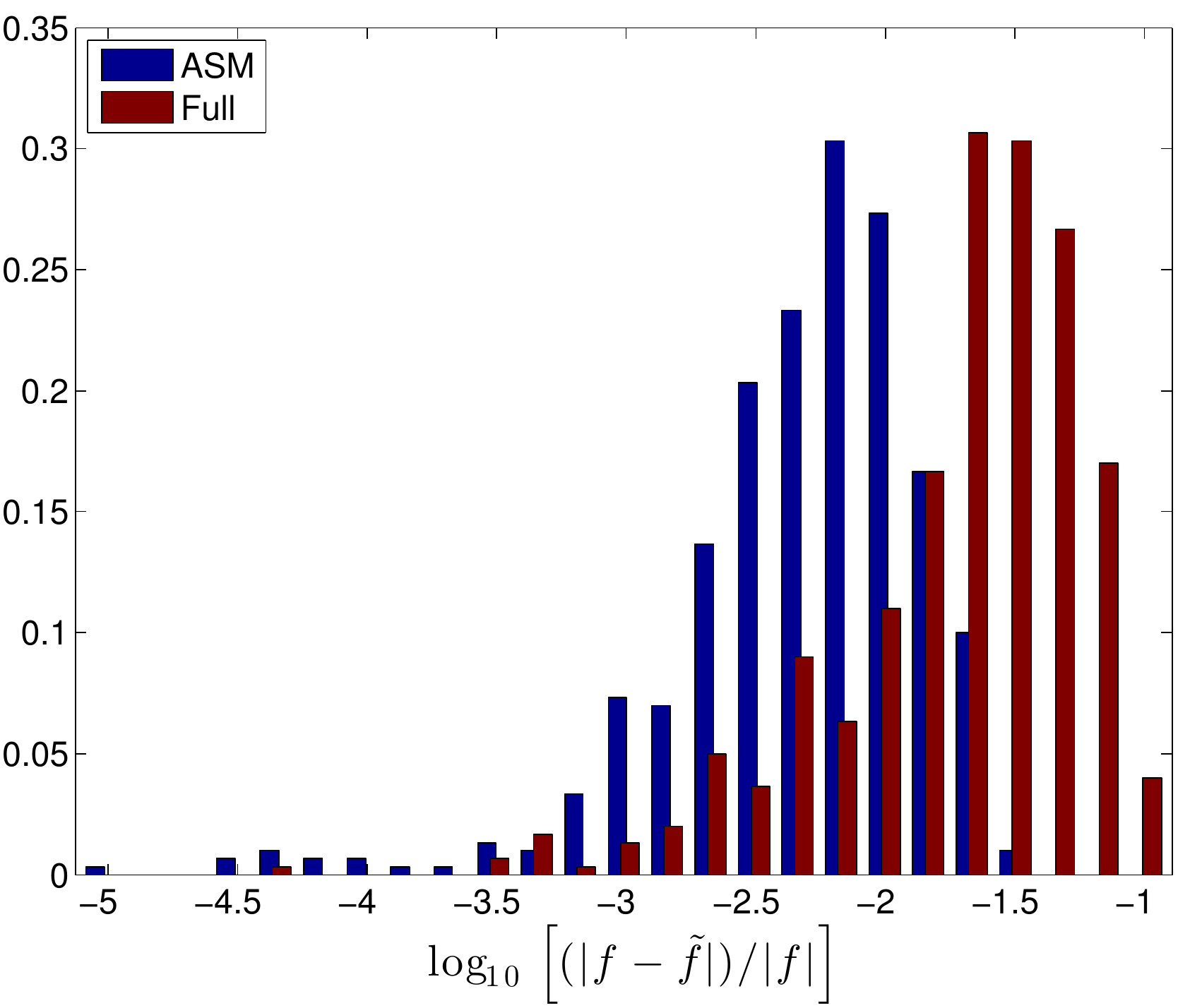}
\label{fig:fullerr4}
}
\caption{Histograms of the log of the relative error in the testing data for kriging surfaces constructed on the one- and two-dimensional subspaces for the two correlation lengths. In the legends, ``ASM'' is the kriging surface using the active subspace, and ``Full'' is the kriging surface on the full space. (Colors are visible in the electronic version.) }
\label{fig:fullerr}
\end{figure}

%% file: sec5-conclusion.tex
\section{Summary \& conclusions}
\label{sec:conclusions}

Active subspace methods enable response surface approximations of a multivariate function on a low-dimensional subspace of the domain. We have analyzed a sequence of approximations that exploits the active subspace: a best approximation via conditional expectation, a Monte Carlo approximation of the best approximation, and a response surface trained with a few Monte Carlo estimates. We have used these analyses to motivate a computational procedure for detecting the directions defining the subspace and constructing a kriging surface on the subspace. We have applied this procedure to an elliptic PDE problem with a random field model for the coefficients. We compared the active subspace method with an approach based on the local sensitivity analysis and showed the superior performance of the active subspace method.

Loosely speaking, active subspace methods are appropriate for certain classes of functions that vary primarily in low-dimensional subspaces of the input. If there is no decay in the eigenvalues of $\mC$, then the methods will perform poorly; constructing such functions is not difficult. However, we have found many high-dimensional applications in practice where the eigenvalues do decay quickly, and the functions respond well to active subspace methods~\cite{Chen2011,Dow2013,Constantine11c,Sensitive12}. Most of those applications look similar to the one presented in Section \ref{sec:example}, where uncertainty in some spatially varying physical input can be represented by a series expansion, and the coefficients of the expansion are treated as random variables; such models arise frequently in UQ.

The computational method we have proposed is ripe for improvements and extensions. We have mentioned many such possibilities in Section \ref{sec:steps}, and we are particularly interested in methods for using fewer evaluations of the gradient to compute the directions defining the active subspace. We will also pursue strategies that make better use of the function evaluations acquired during the gradient sampling.

%% file: isd-rev.bbl
\begin{thebibliography}{10}

\bibitem{Antoulas2005}
{\sc Athanasios~C Antoulas}, {\em Approximation of large-scale dynamical
  systems}, vol.~6, Society for Industrial and Applied Mathematics, 2005.

\bibitem{babuska2004galerkin}
{\sc Ivo Babuska, Ra{\'u}l Tempone, and Georgios~E Zouraris}, {\em Galerkin
  finite element approximations of stochastic elliptic partial differential
  equations}, SIAM Journal on Numerical Analysis, 42 (2004), pp.~800--825.

\bibitem{Bebendorf2003}
{\sc Mario Bebendorf}, {\em A note on the poincar{\'e} inequality for convex
  domains}, ZEITSCHRIFT FUR ANALYSIS UND IHRE ANWENDUNGEN, 22 (2003),
  pp.~751--756.

\bibitem{bliznyuk2008bayesian}
{\sc Nikolay Bliznyuk, David Ruppert, Christine Shoemaker, Rommel Regis, Stefan
  Wild, and Pradeep Mugunthan}, {\em Bayesian calibration and uncertainty
  analysis for computationally expensive models using optimization and radial
  basis function approximation}, Journal of Computational and Graphical
  Statistics, 17 (2008).

\bibitem{Bryson75}
{\sc Arther~E. Bryson and Yu-Chi Ho}, {\em Applied Optimal Control:
  Optimization, Estimation, and Control}, Hemisphere Publishing Corportation,
  1975.

\bibitem{Cai10}
{\sc Jian-Feng Cai, Emmanuel~J. Candes, and Zuowei Shen}, {\em A singular value
  thresholding algorithm for matrix completion}, SIAM Journal on Optimization,
  20 (2010), pp.~1956--1982.

\bibitem{chan1994discussion}
{\sc Kung~Sik Chan and Charles~J Geyer}, {\em Discussion: {Markov} chains for
  exploring posterior distributions}, The Annals of Statistics, 22 (1994),
  pp.~1747--1758.

\bibitem{Chen2011}
{\sc Han Chen, Qiqi Wang, Rui Hu, and Paul Constantine}, {\em Conditional
  sampling and experiment design for quantifying manufacturing error of a
  transonic airfoil}, AIAA-2011-658,  (2011).

\bibitem{chen1982inequality}
{\sc Louis~HY Chen}, {\em An inequality for the multivariate normal
  distribution}, Journal of Multivariate Analysis, 12 (1982), pp.~306--315.

\bibitem{chib1995}
{\sc Siddhartha Chib and Edward Greenberg}, {\em Understanding the
  {Metropolis-Hastings} algorithm}, The American Statistician, 49 (1995),
  pp.~327--335.

\bibitem{Cohen11}
{\sc Albert Cohen, Ingrid Daubechies, Ronald DeVore, Gerard Kerkyacharian, and
  Dominique Picard}, {\em Capturing ridge functions in high dimensions from
  point queries}, Constructive Approximation, pp.~1--19.
\newblock 10.1007/s00365-011-9147-6.

\bibitem{Constantine11c}
{\sc P.~G. Constantine, Q.~Wang, A.~Doostan, and G.~Iaccarino}, {\em A
  surrogate-accelerated {Bayesian} inverse analysis of the {HyShot II} flight
  data}, AIAA-2011-2037,  (2011).

\bibitem{Sensitive12}
{\sc Paul~G. Constantine, Qiqi Wang, and Gianluca Iaccarino}, {\em A method for
  spatial sensitivity analysis}.
\newblock Center for Turbulence Research, Annual Brief, 2012.

\bibitem{doucet2001}
{\sc Arnaud Doucet, Nando De~Freitas, Neil Gordon, et~al.}, {\em Sequential
  Monte Carlo methods in practice}, vol.~1, Springer New York, 2001.

\bibitem{Dow2013}
{\sc Eric Dow and Qiqi Wang}, {\em Output based dimensionality reduction of
  geometric variability in compressor blades}, AIAA-2013-0420,  (2013).

\bibitem{Fornasier2012}
{\sc Massimo Fornasier, Karin Schnass, and Jan Vybiral}, {\em Learning
  functions of few arbitrary linear parameters in high dimensions}, Foundations
  of Computational Mathematics, 12 (2012), pp.~229--262.

\bibitem{fukuda2004zonotope}
{\sc Komei Fukuda}, {\em From the zonotope construction to the minkowski
  addition of convex polytopes}, Journal of Symbolic Computation, 38 (2004),
  pp.~1261--1272.

\bibitem{gill1981practical}
{\sc Philip~E Gill, Walter Murray, and Margaret~H Wright}, {\em Practical
  optimization}, Academic press, 1981.

\bibitem{giunta2006promise}
{\sc AA~Giunta, JM~McFarland, LP~Swiler, and MS~Eldred}, {\em The promise and
  peril of uncertainty quantification using response surface approximations},
  Structures and Infrastructure Engineering, 2 (2006), pp.~175--189.

\bibitem{Griewank00}
{\sc Andreas Griewank}, {\em Evaluating Derivatives: Principles and Techniques
  of Algorithmic Differentiation}, SIAM, 2000.

\bibitem{Halko2011}
{\sc Nathan Halko, Per-Gunnar Martinsson, and Joel~A Tropp}, {\em Finding
  structure with randomness: Probabilistic algorithms for constructing
  approximate matrix decompositions}, SIAM review, 53 (2011), pp.~217--288.

\bibitem{Jameson1988}
{\sc Antony Jameson}, {\em Aerodynamic design via control theory}, Journal of
  scientific computing, 3 (1988), pp.~233--260.

\bibitem{Jolliffe2002}
{\sc I.T. Jolliffe}, {\em Principal Component Analysis}, Springer Verlag,
  2nd~ed., 2002.

\bibitem{Jones2001}
{\sc Donald~R Jones}, {\em A taxonomy of global optimization methods based on
  response surfaces}, Journal of global optimization, 21 (2001), pp.~345--383.

\bibitem{Koehler1996}
{\sc JR~Koehler and AB~Owen}, {\em Computer experiments}, Handbook of
  statistics, 13 (1996), pp.~261--308.

\bibitem{Li2010}
{\sc Jing Li and Dongbin Xiu}, {\em Evaluation of failure probability via
  surrogate models}, Journal of Computational Physics, 229 (2010),
  pp.~8966--8980.

\bibitem{Lieberman10}
{\sc Chad Lieberman, Karen Willcox, and Omar Ghattas}, {\em Parameter and state
  model reduction for large-scale statistical inverse problems}, SIAM Journal
  on Scientific Computing, 32 (2010), pp.~2523--2542.

\bibitem{YeLP08}
{\sc David~G. Luenberger and Yinyu Ye}, {\em LInear and Nonlinear Programming},
  Springer, 2008.

\bibitem{Marzouk2007}
{\sc Youssef~M Marzouk, Habib~N Najm, and Larry~A Rahn}, {\em Stochastic
  spectral methods for efficient {Bayesian} solution of inverse problems},
  Journal of Computational Physics, 224 (2007), pp.~560--586.

\bibitem{Owen13}
{\sc A.~Owen}, {\em Variance components and generalized sobol' indices},
  SIAM/ASA Journal on Uncertainty Quantification, 1 (2013), pp.~19--41.

\bibitem{owenmc-2013}
{\sc Art~B. Owen}, {\em Monte Carlo theory, methods and examples}, 2013.
\newblock \url{http://www-stat.stanford.edu/~owen/mc}.

\bibitem{persson2004simple}
{\sc Per-Olof Persson and Gilbert Strang}, {\em A simple mesh generator in
  matlab}, SIAM review, 46 (2004), pp.~329--345.

\bibitem{Rasmussen2006}
{\sc Carl~Edward Rasmussen and Christopher~KI Williams}, {\em Gaussian
  processes for machine learning}, vol.~1, MIT press Cambridge, MA, 2006.

\bibitem{Russi2010}
{\sc Trent~M. Russi}, {\em Uncertainty Quantification with Experimental Data
  and Complex System Models}, PhD thesis, UC Berkeley, 2010.

\bibitem{Saad1992}
{\sc Youcef Saad}, {\em Numerical methods for large eigenvalue problems},
  vol.~158, SIAM, 1992.

\bibitem{Saltelli2008}
{\sc Andrea Saltelli, Marco Ratto, Terry Andres, Francesca Campolongo, Jessica
  Cariboni, Debora Gatelli, Michaela Saisana, and Stefano Tarantola}, {\em
  Global sensitivity analysis: the primer}, Wiley-Interscience, 2008.

\bibitem{Sirovich87a}
{\sc L.~Sirovich}, {\em Turbulence and the dynamics of coherent structures,
  i-iii}, Quart. Appl. Math., 45 (1987), pp.~561--590.

\bibitem{Wendland2005}
{\sc Holger Wendland}, {\em Scattered data approximation}, vol.~2, Cambridge
  University Press Cambridge, 2005.

\bibitem{wild2008orbit}
{\sc Stefan~M Wild, Rommel~G Regis, and Christine~A Shoemaker}, {\em Orbit:
  Optimization by radial basis function interpolation in trust-regions}, SIAM
  Journal on Scientific Computing, 30 (2008), pp.~3197--3219.

\bibitem{Williams1991}
{\sc David Williams}, {\em Probability with martingales}, Cambridge university
  press, 1991.

\bibitem{Zhou1996}
{\sc Kemin Zhou, John~Comstock Doyle, Keith Glover, et~al.}, {\em Robust and
  optimal control}, vol.~40, Prentice Hall, 1996.

\end{thebibliography}
